\documentclass{amsart}

\usepackage[utf8]{inputenc}
\usepackage[T1]{fontenc}
\usepackage[english]{babel}

\usepackage{fancyhdr}
\usepackage{graphicx}
\usepackage{float}
\usepackage{hyperref}
\usepackage{todonotes}
\usepackage{amsmath}
\usepackage{amssymb}
\usepackage{amsfonts}
\usepackage{amsthm}
\usepackage{mathrsfs}
\usepackage{verbatim}
\usepackage{breqn}
\usepackage{soul}
\usepackage{tikz-cd}
\usepackage{mathtools}
\usepackage{extpfeil}
\usepackage{caption}
\usepackage{subcaption}
\usepackage{extpfeil}
\usetikzlibrary{decorations.pathreplacing,angles,quotes}

\newtheorem{mainthm}{Theorem}

\newtheorem{mainthmTwo}{Theorem}

\newtheorem{thm}{Theorem}
\numberwithin{thm}{subsection}
\numberwithin{equation}{section}

\newtheorem{lem}[thm]{Lemma}
\newtheorem{rmk}[thm]{Remark}
\newtheorem{prop}[thm]{Proposition}

\newtheorem{cor}[thm]{Corollary}
\theoremstyle{definition}

\newcommand{\mult}{\mathrm{mult}}
\newcommand{\lcm}{\mathrm{lcm}}

\newcommand{\C}{\mathbb{C}}
\newcommand{\Ct}{\mathbb{C}{(\!(\!\hspace{0.7pt}t\hspace{0.7pt}\!)\!)}}
\newcommand{\Cu}{\mathbb{C}{(\!(\!\hspace{0.7pt}u\hspace{0.7pt}\!)\!)}}
\newcommand{\D}{\mathbb{D}}
\newcommand{\Q}{\mathbb{Q}}

\newcommand{\X}{\mathscr{X}}

\renewcommand{\O}{\mathcal{O}}
\newcommand{\an}{\mathrm{an}}

\newcommand{\R}{\mathbb{R}}
\newcommand{\Log}{\mathrm{Log}}
\newcommand{\hyb}{\mathrm{hyb}}
\newcommand{\N}{\mathbb{N}}

\renewcommand{\L}{\mathcal{L}}

\newcommand{\Spec}{\mathrm{Spec \ }}

\renewcommand{\div}{\mathrm{div}}
\newcommand{\red}{\mathrm{red}}

\newcommand{\Y}{\mathscr{Y}}

\renewcommand{\tilde}{\widetilde}
\newcommand{\I}{\mathbf{I}}

\newcommand{\CC}{\mathrm{CC}}

\title{Convergence of Bergman measures towards the Zhang measure}
\author{Sanal Shivaprasad}
\date{\today}

\begin{document}
\maketitle

\begin{abstract}
We prove a folklore conjecture that the Bergman measure along a holomorphic family of curves parametrized by the punctured unit disk converges to the Zhang measure on the associated Berkovich space. The convergence takes place on a Berkovich  hybrid space. We also study the convergence of the Bergman measure to a measure on a metrized curve complex in the sense of Amini and Baker. 
\end{abstract}

\section{Introduction}
Any compact Riemann surface $Y$ of genus $g \geq 1$ carries a canonical measure  called the \emph{Bergman measure}, defined as follows. Note that there is a positive definite Hermitian metric on $H^0(Y,\Omega_Y)$, the $g$-dimensional complex vector space of holomorphic 1-forms on $Y$, given by
$$ \langle \phi, \psi \rangle = \frac{i}{2}\int_Y \phi \wedge \overline{\psi}.$$
Pick an orthonormal basis $\phi_1,\dots,\phi_g$ with respect to the above pairing. Then, the positive $(1,1)$-form defined by $\frac{i}{2}\sum_{i=1}^g \phi_i \wedge \overline{\phi_i}$ does not depend on the choice of the orthonormal basis and is called the \emph{Bergman metric} on $Y$ and the associated measure on $Y$ is called as the Bergman measure on $Y$. We can also define the Bergman metric on $Y$ as the pullback of the flat metric from the Jacobian of $Y$ along the Abel-Jacobi map.

The Bergman measure has many applications. For example, the variation of the Bergman measure gives rise to a metric on the Teichm{\"u}ller space of genus $g$ curves for $g \geq 2$ that is invariant under the action of the  mapping class group  \cite{HJ98}.

We would like to understand how the Bergman measure varies in a degenerating holomorphic family of curves. %
This variation has been studied in the literature in many cases \cite{HJ96} \cite{Don15} \cite{dJon19}. 
In this paper, we would like to give the variation of Bergman measures a non-Archimedean interpretation. 

Let $X$ be a complex surface with a holomorphic submersion $X \to \D^*$ with fibers being compact complex curves of genus at least 1. For $t \in \D^*$, let $\mu_t$ denote the Bergman measure on the fiber $X_t$. We would like to understand the convergence of the measures $\mu_t$ on the family $X$.

On $X$, the measures $\mu_t$ converge weakly to the zero measure as $t \to 0$ because there is no fiber over the puncture. This is not very interesting. A way to remedy this issue would be to add a fiber over the puncture to compactify the family over the origin.

One such partial compactification is the `hybrid space', $X^\hyb$, constructed by Berkovich in \cite{Ber09}. 
Here the \emph{central fiber} (i.e.~the fiber over $t=0$) is  $X_{\Ct}^\an$, the Berkovich space obtained by the analytification of $X$ viewed as a variety over the Laurent series field $\Ct$ \cite{Ber90}. The hybrid space has since then been used to study non-Archimedean degenerations \cite{BJ17}, \cite{Oda17}, \cite{Sus18}, \cite{LS19}, \cite{Sch19}, \cite{Shi19} and problems in dynamics \cite{Fav18}, \cite{DF19}, \cite{DKY19}.
When $X$ is a family of curves, the associated Berkovich space can be seen as an inverse limit of a certain family of metric graphs.
These graphs are dual graphs of normal crossing models of $X$. 
For more details, see Section \ref{subsecDualGraph}.

The \emph{Zhang measure} on a metric graph is a weighted sum of Lebesgue measures on edges and Dirac masses on vertices. It was introduced by Zhang in \cite{Zha93} to define a non-Archimedean analogue of the Bergman pairing on a Riemann surface. 
The Zhang measure has been used in the study of potential theory on the Berkovich projective line \cite{BR10}.
The weight of the Zhang measure on an edge is a function involving the length of the edge and the resistance across the endpoints after removing the the edge from the graph. The weight of the Zhang measure on a vertex is the genus of the irreducible component associated to it. 

The Zhang measures on the dual graphs of all normal crossing models of $X$ are compatible and thus give rise to a measure on $X_\Ct^\an$. 

There are several reasons to believe that the Zhang measure is the non-Archimedean analogue of the Bergman measure. 
Firstly, the Weierstrass points on a Riemann surface are equidistributed with respect to the Bergman measure \cite{Nee84}. It is possible to define Weierstrass points on a Berkovich curve or on a  metric graph and it turns out that they are equidistributed with respect to the Zhang measure \cite{Ami14},\cite{Ric18}. Secondly, recall that the Bergman measure can be obtained as a pullback of the flat metric from the Jacobian under the Abel-Jacobi map. Similarly, the Zhang measure can be realized as the pullback of a certain canonical metric on the tropical Jacobian under the tropical Abel-Jacobi map \cite{BF11}.
Thirdly, a version of Kazhdan's theorem for the Bergman measure on a Riemann surface is true for the Zhang measure on a metric graph \cite{SW19}. 

Indeed, it is a folklore conjecture that the Bergman measure converges to the Zhang measure in the hybrid space setting. For example, see \cite[Section 1.1]{SW19}. Our main result gives a positive answer to this conjecture. 
\begin{mainthm}
\label{mainThm1}
The Bergman measure $\mu_t$ on the fiber $X_t$ converges weakly to a measure  $\mu_0$ on the Berkovich space $X_\Ct^\an$, where the convergence takes place on the hybrid space $X^\hyb$.
The measure $\mu_0$ is supported on a subspace of $X^\an_\Ct$ that is isomorphic to a metric graph, and is a weighted sum of Lebesgue measures on edges and Dirac masses on points. 

Moreover, if we assume that $X$ has a semistable model, then $\mu_0$ is the Zhang measure on the Berkovich space $X_\Ct^\an$
\end{mainthm}

In the above theorem, the existence of a semistable model is asking for a normal crossing model $\X$ of $X$ such that $\X_0$ is reduced. Such a model always exists after performing a finite base change $\D^* \to \D^*$ given by $u \mapsto t^n$.

O. Amini has informed us that he and N. Nicolussi have been able to obtain a version of Theorem \ref{mainThm1} using techniques from Hodge theory. 

A key step involved in the proof of Theorem \ref{mainThm1} is to prove the convergence on the hybrid space $\X^\hyb = X \sqcup \Gamma_\X$, associated to a fixed normal crossing model $\X$ of $X$. See Section \ref{secHybridSpace} for details on the topology of the space $\X^\hyb$. 

\begin{mainthmTwo}
\label{mainThm2}
Suppose that $X$ has semistable reduction and let $\X$ be a normal crossing model of $X$. On the space $\X^\hyb$, the measures $\mu_t$ converge weakly to the Zhang measure on $\Gamma_\X$.
\end{mainthmTwo}

We are also able to prove a convergence statement on a hybrid space which has the \emph{metrized curve complex} in the sense of Amini and Baker \cite{AB15} as the central fiber. The metrized curve complex associated to a normal crossing model $\X$ of $X$ is a topological space obtained by replacing each nodal point in $\X_0$ by a line segment. We get $\X_0$ from the associated metrized curve complex by collapsing the line segments. We also get the dual graph $\Gamma_\X$ by collapsing the Riemann surfaces in the metrized curve complex to points. We construct a hybrid space $\X_\CC^\hyb$ which is a partial compactification of $X$ with the central fiber the metrized curve complex associated to $\X$.
\begin{mainthm}
\label{mainThm3}
Assume that $X$ has semistable reduction and let $\X$ be a normal crossing model of $X$. Then, there exists a measure $\mu_\CC$ on the metrized curve complex associated to $\X$ such that $\mu_t$ converges weakly to $\mu_\CC$ as $t \to 0$, when seen as measures on $\X^\hyb_\CC$.
\end{mainthm}

The measure $\mu_\CC$ restricted to each Riemann surface of positive genus in the metrized curve complex is exactly the Bergman measure on that Riemann surface. The measure $\mu_\CC$ places no mass on any genus zero Riemann surface in the metrized curve complex. The restriction  of $\mu_\CC$ on an edge is exactly the Zhang measure restricted to the edge. This shows us that the Dirac masses that show up in the Zhang measure correspond to collapsed Bergman measures.


Theorem \ref{mainThm2} is closely related to \cite[Remark 16.4]{dJon19}. The main difference between the two results is that  \cite{dJon19} does not involve any Berkovich spaces and the limiting measure lives on the singular curve $\X_0$ while in our case the limiting measure is on the metric graph $\Gamma_\X$. Another difference is that de Jong's result only applies to semistable models of $X$ while we also deal with the case when the central fiber is not necessarily reduced. 
The limiting measure in \cite[Remark 16.4]{dJon19} is the sum of the Bergman measures on the normalization of positive genus irreducible components of $\X_0$ and some Dirac masses on nodal points. The mass at a nodal point is equal to the total mass of the corresponding edge in the Zhang measure. Theorem \ref{mainThm3} serves as a concrete link between the two: the pushforward of $\mu_\CC$ to $\X_0$ gives the limiting measure in \cite[Remark 16.4]{dJon19} while its pushforward to $\Gamma_\X$ gives the Zhang measure. So we recover both Theorem \ref{mainThm2} and \cite[Remark 16.4]{dJon19} from Theorem \ref{mainThm3} by considering the continuous maps $\X^\hyb_\CC \to \X^\hyb$ and $\X^\hyb_\CC \to \X$.

To prove Theorem \ref{mainThm1} using Theorem \ref{mainThm2}, we just need to show that the convergence given by Theorem \ref{mainThm2} for different models are compatible i.e.~if $\X,\X'$ are models of $X$ such that we have a proper map $\X' \to \X$ which restricts to identity on $X$, then the limiting measures seen as measures on $\Gamma_{\X'}$ using $\Gamma_\X \hookrightarrow \Gamma_{\X'}$ are the same. Now, using the fact that $X^\hyb = \varprojlim_{\X} \X^\hyb$, we get Theorem \ref{mainThm1} in the case when $X$ has semistable reduction. Since a semistable reduction always exists after a base change, to prove Theorem \ref{mainThm1} in general, we only need to understand what happens after a base change.  

To prove Theorem \ref{mainThm2} on $\X^\hyb$ for a normal crossing model $\X$ of $X$, we make a careful choice of elements of $H^0(\X,\Omega_{\X/\D})$ that restrict to a basis of $H^0(X_t,\Omega_{X_t})$ for all $t$ and also to good basis of $H^0(\X_{0,\red},\omega_{\X_{0,\red}})$. We also work with $\X_{0,\red}$ instead of $\X_0$ because the dualizing sheaf, $\omega_{\X_{0,\red}}$, is better behaved. We express the Bergman measure in terms of this basis and compute some asymptotics. Our analysis strongly uses the analogy between one-forms on Riemann surfaces and on metric graphs.

To prove Theorem \ref{mainThm3}, we first construct the metrized curve complex hybrid space $\X^\hyb_\CC$ for a normal crossing model $\X$ of $X$. We then analyze the convergence in a small enough neighborhood of each point in the central fiber. For non-nodal points that lie on an irreducible component of $\X_0$ or points in the interior of a line segment, this computation is a minor modification of the computations done to prove Theorem \ref{mainThm2}. So, we only need to study the convergence in a neighborhood of a point that is the intersection of an irreducible component of $\X_0$ and a line segment. The proof of this part uses the same kind of analysis, just a more careful one.

A major difference between the results of \cite{BJ17} and this paper is that  the limiting measure in \cite{BJ17} is either always Lebesgue or always atomic, but never a sum of both. For $g=1$, Theorem \ref{mainThm1} recovers the one-dimensional case of the  convergence theorem in \cite{BJ17}. See also \cite[Corollary 4.8]{CLT10} for a related statement.




We would also like to point out that some of the asymptotics that we use to prove Theorem \ref{mainThm2} are similar to the ones used by de Jong to prove \cite[Remark 16.4]{dJon19}. For example, compare Lemma \ref{lemAsymptotics} and \cite[Equation (16.7)]{dJon19}. De Jong's asymptotics are more versatile as they involve families $\X \to \D^m$ and are proved using the theory of variation of mixed Hodge structures. We don't use any variation of mixed Hodge structures and prove these asymptotics for $m=1$ by explicit computations.

\subsection{An example}
\begin{figure}
\centering
\begin{subfigure}{0.45\linewidth}
\centering
\includegraphics[scale=0.44]{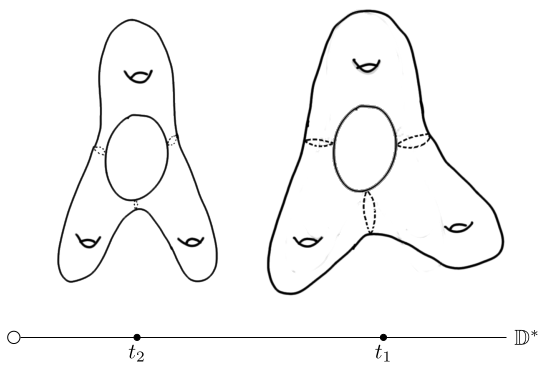}
\caption{A family of genus 4 curves obtained by pinching the dotted lines.}%
\label{subfigGenus4Family}
\end{subfigure}%
\hspace{0.05\linewidth}
\begin{subfigure}{0.45\linewidth}
\centering
\includegraphics[scale=0.44]{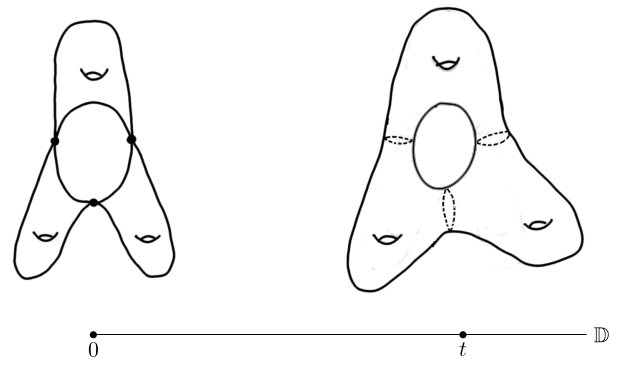}
\caption{The minimal normal crossing model of the family}
\label{subfigNCModel}
\end{subfigure}
\caption{A family of genus 4 curves and its minimal normal crossing model.}
\label{figGeus4FamilyAndNCModel}
\end{figure}

Let $X \to \D^*$ be a family of compact genus 4 Riemann surfaces given by pinching the dotted simple closed curves in Figure \ref{subfigGenus4Family}.
Then, the central fiber of the minimal normal crossing model, $\X$, has three irreducible components each of genus one intersecting at 3 nodal points (see Figure \ref{subfigNCModel}). The associated hybrid space is shown in Figure  \ref{figBerkovichHybridSpace}.

\begin{figure}
  \centering
  \includegraphics[scale=0.4]{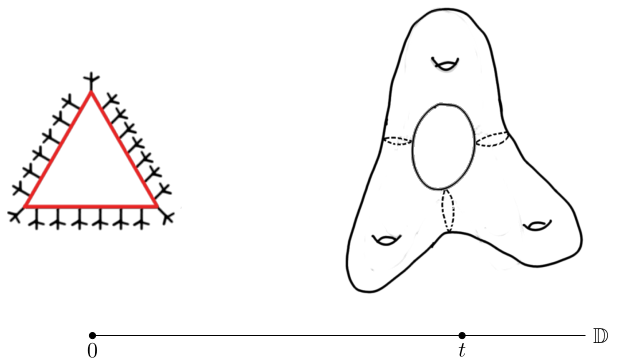}
  \caption{The hybrid space associated to the family in Figure \ref{subfigGenus4Family}. The support of the Zhang measure is shown in red.}
  \label{figBerkovichHybridSpace}
\end{figure}

In this case, the dual graph, $\Gamma_\X$, is a triangle with all three vertices of genus one. The Zhang measure is a sum of a Lebesgue measure on each of edge of mass $\frac{1}{3}$ and a Dirac mass on each vertex of mass $1$. 
The central fiber of the hybrid space has a subspace homeomorphic to $\Gamma_\X$.

\begin{figure}
   \centering
  \includegraphics[scale=0.4]{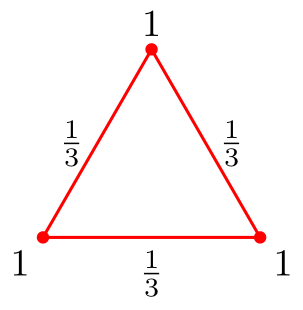}
 \caption{The Zhang measure on the Berkovich space associated to the family in Figure \ref{subfigGenus4Family}}
 \label{figZhangAndPluriZhang}
 \end{figure}

The curve complex hybrid space associated to the minimal normal crossing model is shown in Figure \ref{figCCHybridSpace}. The measure $\mu_\CC$ on the metrized curve complex in the sum of the Bergman (Haar) measures on each of the genus 1 curves and Lebesgue measure of mass $\frac{1}{3}$ on each of the edges. 

\begin{figure}
  \centering
  \includegraphics[scale=0.4]{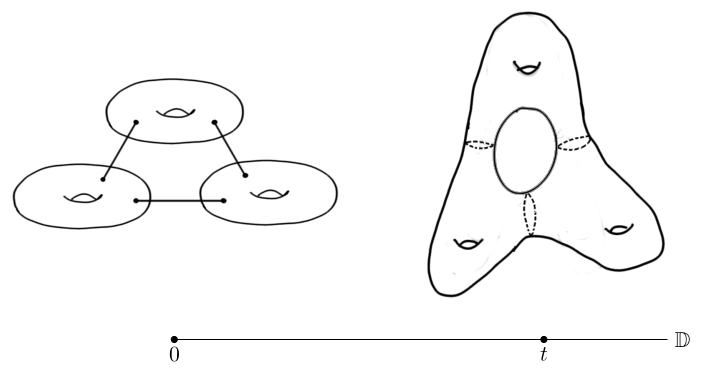}
  \caption{The curve complex hybrid space associated to the family in Figure \ref{subfigNCModel}.}
  \label{figCCHybridSpace}
\end{figure}

\subsection*{Further questions} We will address the convergence of Bergman measures associated to higher tensor powers of the canonical line bundle in future work \cite{Shi20}.

We can define the Bergman measure on higher dimensional complex manifolds and ask how these Bergman measures converge on the hybrid space. We could also ask if there is a $p$-adic analogue of such a convergence in the sense of \cite{JN19}.

\subsection*{Structure of the paper}
In Section \ref{secPreliminaries}, we discuss some preliminaries. In Section \ref{secHybridSpace}, we recall the construction of the hybrid space. In Section \ref{secCanonicalSheaf}, we recall some properties of the dualizing sheaf of curves with at worst simple nodal singularities. In Section \ref{secAsymptotics}, we compute some asymptotics related to the Bergman measure. In Section \ref{secConvergenceTheorem}, we prove Theorems \ref{mainThm2} and \ref{mainThm1}. The key technical result in this section is Lemma \ref{lemKey}. In Section \ref{secCCConvergence}, we work out the convergence on the metrized curve complex hybrid space, proving Theorem \ref{mainThm3}.

\subsection*{Acknowledgments}
I thank Mattias Jonsson for suggesting this problem to me, and also for his support and guidance. I also thank Omid Amini, Sébastien Boucksom, Robin de Jong, Holly Krieger, Yuji Odaka, Harry Richman and Farbod Shokrieh for helpful comments on a preliminary draft. 
This work was supported by the NSF grants DMS-1600011 and DMS-1900025.

\section{Preliminaries}
\label{secPreliminaries}
\subsection{Curves and models}
Throughout this paper, a \emph{family of curves} $X$ over $\D^*$ of genus $g \geq 1$ refers to a complex manifold $X$ of dimension 2 such that we have a smooth projective holomorphic map $X \to \D^*$ with fibers being connected smooth complex projective curves of genus $g$. We also assume that the family is meromorphic at $0$ i.e.~there exists a projective flat family $\X \to \D$ extending $X \to \D^*$ with $\X$ normal and having a non-empty fiber over $0$. 

A \emph{model}  $\X$ of $X$ is a flat projective holomorphic family $\X \to \D$ such that $\X|_{\D^*}$ is biholomorphic to $X$ as spaces over $\D^*$. 
We say that $\X$ is a \emph{regular model} of $X$ if $\X$ is regular.
The fiber over $0$, $\X_0$, is called the special fiber. Let $\X_{0,\red}$ denote the reduced induced structure on $\X_{0}$. 

We say that $\X$ is a \emph{normal crossing model} (abbreviated as \emph{nc model}) of $X$ if $\X$ is regular and $\X_{0,\red}$ is a normal crossing divisor.  

\subsection{Semistable reduction and minimal nc models}
\label{subsecSemistableReductionAndMinimalModel}
We refer the reader to \cite{Rom13} for a detailed introduction to models over a DVR. We summarize some of the results that we will use.

For any model $\X$ of $X$, we always have that $\X_0$ is connected \cite[Corollary 8.3.6]{Liu02}.

A family of curves $X$ is said to have \emph{semistable reduction} if there exists an nc model $\X$ of $X$ with reduced special fiber i.e.~$\X_0 = \X_{0,\red}$ and  such an $\X$ is called a \emph{semistable model} of $X$.

A family of curves $X$ of genus $g \geq 1$ always has semistable reduction  after performing a finite base change $\D^* \to \D^*$ given by $u \mapsto t^n$. This follows from \cite[Corollary 2.7]{DM69} in the case when $g \geq 2$. See
 \cite[\href{https://stacks.math.columbia.edu/tag/0CDN}{Tag 0CDN}]{stacks-project} for a general statement.

A family of curves $X$ of genus $g \geq 1$ always has a \emph{minimal nc model} i.e.~there exists an nc model $\X_{\min}$ of $X$ such that for any nc model $\X$ of $X$, there is a proper morphism $\X \to \X_{\min}$. 
Such a model is unique up to a unique isomorphism. See \cite[Theorem 2.5.1]{Rom13} or \cite[\href{https://stacks.math.columbia.edu/tag/0C6B}{Tag 0C6B}]{stacks-project} for details. 

When $X$ has semistable reduction, the minimal nc model is also semistable \cite[\href{https://stacks.math.columbia.edu/tag/0CDG}{Tag 0CDG}]{stacks-project}. In addition, the special fiber of the minimal nc model has no non-singular rational component that meets the rest of the component in only one point.

\subsection{Blowups and getting new models from old ones}
Given two models $\X$ and $\X'$ of $X$, we say that $\X'$ \emph{dominates} $\X$ and write $\X' \geq \X $ if we have a proper holomorphic map  $\X' \to \X$ such that its restriction to $\X'|_{\D^*}$ commutes with the isomorphism to $X$.

If $\X,\X'$ are two nc models of $X$ such that $\X' \geq \X$, then $\X'$ is obtained from $\X$ by a sequence of blowups at closed points in the special fiber \cite[Theorem 1.15]{Lic68}. 

If $\X$ is an nc model of $X$, we can get a new nc model $\X'$ dominating $\X$ by blowing up $\X$ at a closed point in $\X_0$. 
Given two models $\X$ and $\X'$ of $X$, there always exists a model $\X''$ such that $\X'' \geq \X$, $\X'' \geq \X'$, and $\X''$ is obtained from both $\X$ and $\X'$ by a sequence of blowups in the special fiber \cite[Proposition 4.2]{Lic68}.

\subsection{Dual graph associated to a model}
\label{subsecDualGraph}
Let $\X \to \D$ be an nc model of $X$. The \emph{dual graph} $\Gamma_\X$ associated to $\X$ a connected metric graph . The vertices of $\Gamma_\X$ correspond to the irreducible components of $\X_0$. If $P$ is a node in $\X_0$ that lies in the intersection of the components $E_0$ and $E_1$, then we add an edge $e_p$ between the vertices $v_{E_0}$ and $v_{E_1}$. Let $V(\Gamma_\X)$ and $E(\Gamma_\X)$ denote the vertex and edge set of the dual graph respectively. Note that $\Gamma_\X$ might have loop edges and multiple edges between a pair of vertices.

We define a length on each edge i.e.~we have a function $l : E(\Gamma_\X) \to \Q_{\geq 0}$ defined as follows. Let $z,w$ be the (analytic) local equations defining the irreducible components containing a node $P$. Then, locally near $P$, the map $\X \to \D$ is given by $(z,w) \mapsto z^aw^b$, where $a$ and $b$ are the respective multiplicities of the irreducible components. We define the length of $e_P$  to be $\frac{1}{ab}$.

It is also useful to keep track of the genus of the irreducible components. So our metric graph also comes with the data of a genus function $g : V(\Gamma_\X) \to \N$ given by taking the value of the genus of the normalization of an irreducible component at every vertex. We also define the genus of $\Gamma_\X$ to be its first Betti number i.e.~
$$g(\Gamma_\X) = |V(\Gamma_\X)| - |E(\Gamma_\X)| + 1.$$ 

Note that if $\X$ is a semistable model of $X$, all the edges in the dual graph $\Gamma_\X$ would have length $1$. 
For more details about the dual metric graphs, refer to \cite{BF11}, \cite{BPR13} and \cite{BPR16}.

Let $\X'$ be obtained by blowing up $\X$ at a closed point in $\X_0$. Then, $\Gamma_\X$ and $\Gamma_{\X'}$ are related as follows.
\begin{itemize}
\item If $\X'$ is obtained by blowing up a smooth point on an irreducible component $E_0 \subset \X_0$ of multiplicity $a$, then $\Gamma_{\X'}$ is obtained from $\Gamma_{\X}$ by adding a new vertex $v_{E}$ corresponding to the exceptional divisor of the blowup and adding an edge of length $\frac{1}{a^2}$ between $v_{E}$ and $v_{E_0}$. The genus function is extended to one on $\Gamma_{\X'}$ by defining it to be $0$ on $v_{E}$. 
\item If $\X'$ is obtained by blowing up a node $P = E_1 \cap E_2$ for (possibly same)  irreducible components $E_1,E_2 \subset \X_0$, then $\Gamma_{\X'}$ is obtained from $\Gamma_{\X} $ by subdividing the edge $e_{P}$ into edges of lengths $\frac{1}{a(a+b)}$ and  $\frac{1}{(a+b)b}$ by adding a vertex $v_E$ corresponding to the exceptional divisor. This makes sense as
  $$ \frac{1}{ab} = \frac{1}{a(a+b)} + \frac{1}{(a+b)b} .$$
The genus function is extended to $\Gamma_{\X'}$ by defining it to be 0 on $v_{E}$. 
\end{itemize}
In both the cases, we see that we have an inclusion $\Gamma_\X \hookrightarrow \Gamma_{\X'}$ as well as a retraction $\Gamma_{\X'} \to \Gamma_{\X}$, and thus $\Gamma_{\X'}$ is a deformation retract of $\Gamma_{\X}$. They both also have the `same' genus function.

More generally, given two nc models $\X$ and $\X'$, they can both be dominated by a common model $\X''$ obtained by a sequence of blowups from both $\X$ and $\X'$. Thus, we see that $g(\Gamma_{\X}) = g(\Gamma_{\X'})$ and $\sum_{v \in V(\Gamma_{\X})} g(v) = \sum_{v \in V(\Gamma_{\X'})} g(v)$. Let $g' = g(\Gamma_\X)$.

The following remark is a consequence of the invariance of the genus functions under blowups.
\begin{rmk}
\label{rmkMultiplicities}
Suppose that $X$ has a semistable model and let $\X$ be an nc model of $X$. Then, any irreducible component $E \subset \X_0$ whose normalization has positive genus, has multiplicity $1$.
\end{rmk}

\begin{rmk}[The two choices of the length function]
  There are two possible ways of assigning lengths that we can assign to a node $P$ given by the intersection of two irreducible components of $\X_0$ with multiplicities $a$ and $b$ respectively. One way is to define the lengths as above, by setting $$l_1(e_P) = \frac{1}{ab} .$$
  Yet another way is to define the length by
  $$ l_2(e_P) = \frac{1}{\lcm(a,b)} .$$
  Both these lengths are compatible with respect to blowups. This follows from the fact that
  $$ \frac{1}{\lcm(a,b)} = \frac{1}{\lcm(a,a+b)} + \frac{1}{\lcm(a+b,b)} .$$
  See \cite{BN16} 
  for comparisons between the two metrics. The advantage of using the first length function is that it makes our computations easier and the advantage of using the second one is that it is well-behaved with respect to ground field extensions. 

  In our case, it turns out that we could have chosen either one of the above  metrics and it would not matter. The reason for this is that if we assume that $X$ has a semistable model, the two notions of length can only differ on bridge edges of the dual graphs associated to any model. Since our aim is to compute the Zhang measure on the dual graph using the length function, it is enough to realize that Zhang measure remains invariant under change of length of any bridge edge.  
\end{rmk}

\subsection{The Zhang measure on the dual graph}
\label{subsecZhangMeasure}
Let $\Gamma$ be the metric graph along with a genus function $g : V(\Gamma) \to \N$. The \emph{Zhang measure} on $\Gamma$ is a measure and is given as follows.
$$ \mu_{Zh} = \sum_{v \in V(\Gamma)} g(v) \delta_{v} + \sum_{e \in E(\Gamma)} \frac{dx|_{e}}{l_e + r_e}$$

Here $\delta_v$ is a Dirac measure at $v$, $l_e$ is the length of the edge $e$, $r_e$ is the resistance between the endpoints of the edge $e$ in the graph obtained by removing the edge $e$ and $dx|_{e}$ is the Lebesgue measure on the edge $e$ normalized such that $\int_{e}dx|_{e} = l_e$.  When $e$ is a bridge edge i.e.~removing $e$ from $\Gamma$ disconnects $\Gamma$, then $r_e := \infty$ and $\frac{1}{l_e + r_e} = 0$. Thus, the Zhang measure places no mass on bridge edges. For more details, see \cite{Zha93}. Note that our definition differs from Zhang's original definition by a factor of $g$. This is done so that so that the total mass of Zhang measure is now equal to $g = \sum_{v \in V(\Gamma)} g(v) + g(\Gamma)$. For an interpretation of $\frac{1}{l_e + r_e}$ in terms of spanning trees and electrical networks, refer to \cite[Section 6]{BF11}. 

\begin{rmk}
  Note that the Zhang measure is invariant under the following operations.
  \begin{itemize}
  \item If we subdivide an edge of length $l$ into two edges of lengths $l_1,l-l_1$, the Zhang measure does not change.
  \item If we introduce a new vertex $v'$ and add a new edge $e$ between $v'$ and an existing vertex $v$, the Zhang measure on the new graph is the same as the one on the old graph as the edge $e$ would be a bridge and would not alter any of the resistances in the old graph.    
  \item If we multiply all the lengths by a fixed factor $N \in \R_+$, the Zhang measure does not change. This is because the resistance is linear as a function of edge lengths and thus the quantity $\frac{l_{e}}{l_e + r_e}$ remains unchanged.  
\end{itemize}

The first two operations correspond to altering an nc model by blowups and the third operation corresponds to ground field extension.
  
\end{rmk}

\subsection{Bergman measure on a complex curve}
Let $Y$ be a complex curve of genus $g \geq 1$. Then there exists a natural Hermitian metric on $H^0(Y,\Omega_Y)$ given by
\begin{equation}
\label{eqnHermitianPairing}
(\vartheta,\vartheta') \mapsto \frac{i}{2} \int_Y \vartheta \wedge \overline{\vartheta'}.
\end{equation}
Let $\vartheta_1,\dots,\vartheta_g$ be an orthonormal basis of $H^0(Y,\Omega_Y)$ with respect to this pairing. Then, we get a positive $(1,1)$-form $\frac{i}{2}\sum_{i} \vartheta_i \wedge \overline{\vartheta_i}$ on $Y$. It is easy to verify that this $(1,1)$-form does not depend on the choice of the orthonormal basis.
This $(1,1)$-form gives rise to a measure on $Y$ which is known as the \emph{Bergman measure}. Note that the total mass of the Bergman measure is $g$. For more details regarding the Bergman measure, see \cite{Ber10} and \cite[Section 3.3]{BSW19}.

\subsection{Associated Berkovich space}
The \emph{Berkovich space}, $Y^\an$, associated to a proper variety $Y$ defined over a non-Archimedean field $K$ is a compact Hausdorff topological space. As a set, $Y^\an$ consists of pairs $(x,v)$ where  $x \in Y$  is a (not necessarily closed) point and $v$ is a valuation on $K(x)$ extending the given valuation on $K$.

In our setting, $K = \Ct$ with the $t$-adic valuation. Let $X_\Ct$ be the projective variety cut out by the defining equations of $X$, where we view the coefficients of the defining polynomial as elements of $\Ct$ by looking at the power series expansion around $0$.

The collection of all nc models of $X$ forms a directed system. Given a proper morphism $\X' \to \X$, we get a retraction map $\Gamma_{\X'} \to \Gamma_\X$. For example, if $\X'$ is obtained by blowing up $\X$ at a node in $\X_0$, then this map is an isometry and if $\X'$ is obtained by blowing up $\X$ at a smooth point $P \subset \X_0$, then this map is obtained by collapsing the vertex and edge associated to the exceptional divisor and the new node respectively to the vertex on $\Gamma_\X$ associated to the irreducible component containing $P$. More generally, see \cite{MN15} 
for a description of this map. 

Then, we have an homeomorphism \cite[Theorem 10]{KS06} \cite[Corollary 3.2]{BFJ16} 
$$X_\Ct^\an \simeq \varprojlim_{\text{nc models }\X} \Gamma_\X. $$ 

A reader unfamiliar with Berkovich spaces may take the above as the definition of the associated Berkovich space, as we will mostly be using this description.

\section{The hybrid space}
\label{secHybridSpace}
Given an nc model $\X$ of $X$, we can construct a \emph{hybrid space} $\X^\hyb$ given set theoretically as $\X^\hyb = X \sqcup \Gamma_{\X}$. 
We can topologize the hybrid space as follows. We refer the reader to \cite{BJ17} and \cite{Shi19} for a more detailed discussion regarding the construction of the hybrid space.

Consider a chart given by an open subset $U \subset \X$ such that $U \cap \X_0 = U \cap E$, where $E$ is an irreducible component of $\X_0$ of multiplicity $a$ and there exist coordinates $(z,w)$ on $U$ with $|z|,|w| < 1$  such that the projection to $\D$ is given by $(z,w) \mapsto z^a$. Following the terminology of \cite{BJ17}, we call such a coordinate chart as being \emph{adapted} to $E$. In this case, we define $\Log_U : U \setminus E \to v_{E}$ to be the constant function, where $v_E \in \Gamma_{\X}$ is the vertex corresponding to $E$. 

Now, let $P = E_1 \cap E_2$ be a node where $E_1$ and $E_2$ are either two distinct irreducible components of $\X_0$, or correspond to two different local analytic branches of the same irreducible component. Let the multiplicities of $E_1,E_2$ in $\X_0$ be $a,b$ respectively. Now consider a coordinate chart given by an open set $U \subset \X$ such that $U \cap \X_0$ = $U \cap (E_1 \cup E_2)$ and there exist coordinates $(z,w)$ on $E$ with $|z|,|w| < 1$, $U \cap E_1 = \{ z = 0 \}$, $U \cap E_2 = \{ w = 0 \}$ and the projection to the disk is given by $(z,w) \mapsto t = z^aw^b$. Such a coordinate chart is said to be \emph{adapted} to the node $P = E_1 \cap E_2$. In this case, we define $\Log_U : U \setminus \X_0 \to e_{P}$ by $(z,w) \mapsto \frac{\log|z|}{b\log|t|}$, where we identify $e_P$ with $[0,\frac{1}{ab}]$ with $v_{E_2}$ corresponding to $0$ and $v_{E_1}$ corresponding to $\frac{1}{ab}$. 

A coordinate chart adapted to either an irreducible component of $\X_0$, or to a node in $\X_0$ is called such a coordinate chart as an adapted coordinate chart.

Let $V = \bigcup_i U_i$ be a finite cover of an open neighborhood of $\X_0$ by adapted coordinate charts  and let $\chi_i$ be a partition of unity with respect to the cover $U_i$. Then the function $\Log_V : V \setminus \X_0 \to \Gamma_\X$ defined by $\Log_V = \sum_i \chi_i \Log_{U_i}$ is well-defined (note that addition in $\Gamma_\X$ is not well-defined, but it makes sense on an edge using the identification $e_P \simeq [0,l_{e_P}]$). Such a function is called a global log function. The following remark is very useful and is proved using \cite[Theorem 5.7]{Cle77}. 

\begin{rmk}[Proposition 2.1, \cite{BJ17}]
\label{rmkLogsAreSame}
  If $V$ and $W$ are open neighborhoods of $\X_0$ with global log functions $\Log_V$ and $\Log_W$, then as $t \to 0$
  $$ \Log_V - \Log_W = O\left( \frac{1}{\log|t|^{-1}} \right)  $$
  uniformly on compact sets of $V \cap W$ . 
\end{rmk}

We define the topology on $\X^\hyb$  to be the coarsest topology satisfying
\begin{itemize}
\item The map $X \to \X^\hyb$ is an open immersion. 
\item The map $\X^\hyb \to \D$ given by extending $\pi : X \to \D^*$ and sending $\Gamma_\X$ to the origin is continuous.
\item Given a global log function $\Log_V$, the map $V \cup \Gamma_\X \to \Gamma_\X$ given by $\Log_V$ on $V$ and identity on $\Gamma_\X$ is continuous. 
\end{itemize}

It follows from Remark \ref{rmkLogsAreSame} that the topology induced on $\X^\hyb$ does not depend on the choice of the global log function.

Define $X^\hyb$ to be $\varprojlim_\X \X^\hyb$, where $\X$ runs over all normal crossing models. Since we have that $X_\Ct^\an = \varprojlim \Gamma_\X$, we get that the central fiber of $X^\hyb$ is homeomorphic to $X_\Ct^{\an}$. In fact, it is possible to see the space $X^\hyb$ as the Berkovich analytification of $X$ seen as a scheme over a certain Banach ring \cite{Ber09}, \cite[Appendix]{BJ17}. See also \cite{Poi13}.

\section{The canonical sheaf on $\X_{0,\red}$}
\label{secCanonicalSheaf}
If $Y$ is a a smooth projective complex curve, then we define its \emph{dualizing sheaf} $\omega_Y$ as the sheaf of holomorphic de-Rham differentials $\Omega_Y$ i.e.~$\omega_Y = \Omega_Y$. This sheaf satisfies Serre duality i.e.~for any line bundle $\L$ and for $i=0,1$
$$ H^i(Y,\L) \simeq H^{1-i}(Y,\omega_{Y} \otimes_{\O_Y} \L^{\vee})^\vee.$$

In certain more general situations, it is possible to define a sheaf that satisfies similar duality properties. For example, if $Y$ is a Cohen-Macaulay variety, one can define a dualizing sheaf $\omega_{Y}$. \cite[Section III.7]{Har77}

Let $\X$ be an nc model of $X$. A simple computation shows that $\X_{0,\red}$ is a Cohen-Macaulay variety and thus it is possible to define $\omega_{\X_{0,\red}}$. The sheaf $\omega_{\X_0,\red}$ is in fact a line bundle. We give a more explicit description of it later in this section. 

Let us first calculate $\dim_{\C}H^0(\X,\omega_{\X_{0,\red}})$. Since $\omega_{\X_{0,\red}}$ is the dualizing sheaf of $\X_{0,\red}$, by applying Serre duality we get that $$H^0(\X,\omega_{\X_{0,\red}}) \simeq H^1(\X_{0,\red},\O_{\X_{0,\red}})^\vee.$$
Let $\tilde{\X_{0,\red}}$ denote the normalization of $\X_{0,\red}$ and let $p : \tilde{\X_{0,\red}} \to \X_{0,\red}$ denote the normalization map. Then, $\tilde{\X_{0,\red}}$ is a possibly disconnected union of curves. By looking at the long exact sequence induced in cohomology by $$ 0 \to \O_{\X_{0,\red}} \to p_*(\O_{\tilde{\X_{0,\red}}}) \to \sum_{P \in \X_{0,\red} \text{ node}} \C(P) \to 0,$$ it follows that $$\dim_\C H^1(\X_{0,\red},\O_{\X_{0,\red}}) = g(\Gamma_{\X}) + \sum_{v \in V(G)}g(v).$$

If $\X$ is a semistable model of $X$, then $\X_0 = \X_{0,\red}$ and the invariance of the arithmetic genus in flat families guarantees that $\dim_\C H^1(\X_0,\O_{\X_0}) = g$. Thus, in this case, we see that
$g = g(\Gamma_\X) + \sum_{v \in V(\Gamma_\X)} g(v)$.

More generally, the same holds true for any model $\X$ as long as we assume that $X$ has a semistable model. This follows from the fact that $ g(\Gamma_{\X}) + \sum_{v \in V(\Gamma_{\X})}g(v)$ does not depend on the choice of the nc model. (See Section \ref{subsecDualGraph}). In this case, it also follows that $\dim_\C H^0(\X_{0,\red},\omega_{\X_{0,\red}}) = g$. 

\subsection{An explicit description of $\omega_{\X_{0,\red}}$}
\label{subsecOmegaX0Red}
It is possible to give an explicit description of the elements of $H^0(\X_{0,\red},\omega_{\X_{0,\red}})$: they correspond to  meromorphic $1$-forms $\psi$ on $\tilde{\X_{0,\red}}$, the normalization of $\X_{0,\red}$, with at worst simple poles at the points that lie above the nodes in $\X_0$ such that if $P'$ and $P''$ lie above the node $P$, then the residues of $\psi$ at $P'$ and $P''$ add up to 0 \cite[Section I]{DM69}.

Let $E_1,\dots,E_m$ denote the irreducible components of $\X_{0,\red}$.  Then, note that $\tilde{\X_{0,\red}} = \bigsqcup_i \tilde{E_i}$, where $\tilde{E_i}$ is the normalization of $E_i$. When $E_i$ does not have a self-node, then $E_i = \tilde{E_i}$.

Let $P^{(i)}_1,\dots,P^{(i)}_{r_i}$ denote the points in $\tilde{E_i}$ that lie over nodal points in $\X_0$. The above description gives rise to the following short exact sequence of sheaves on $\X_{0,\red}$: 
$$ 0 \to \omega_{\X_{0,\red}} \to \bigoplus_i \omega_{\tilde{E}_i}(P^{(i)}_1+\dots+P^{(i)}_{r_i})  \to \bigoplus_{P \in \X_0 \text{ node }} \C(P) \to 0,$$

where the first map is given by the restrictions $\psi \mapsto (\psi|_{\tilde{E}_1},\dots,\psi|_{\tilde{E}_m})$ and the second map is given by taking the sum of residues. 

We also have a natural inclusion $\omega_{\tilde{\X_{0,\red}}} = \bigoplus_{i} \omega_{\tilde{E}_i} \hookrightarrow \omega_{\X_{0,\red}} $ as the sections of $\omega_{\tilde{\X_{0,\red}}}$ have zero residue at all points. Since $H^0(\tilde{\X_{0,\red}},\omega_{\tilde{\X_{0,\red}}})$, the vector space of holomorphic 1-forms on $\tilde{\X_{0,\red}}$, has dimension $\sum_{v \in V(\Gamma)}g(v)$, it follows that the subspace of $H^0(\X_{0,\red}, \omega_{\X_{0,\red}})$ spanned by $1$-forms that have no poles has dimension $\sum_{v \in V(\Gamma)}g(v) = g - g(\Gamma_\X)$. 

\subsection{One-forms on metric graphs}
We refer the reader to \cite[Section 2.1]{BF11} for a detailed introduction to one-forms on metric graphs. 
Let $\Gamma$ be a connected metric graph of genus $g'$. Assume that $\Gamma$ is oriented i.e.~a choice of an orientation for each edge of $\Gamma$. Then 
we define the space of \emph{one-forms} on $\Gamma$ as:
$$ \Omega(\Gamma) = \left\{ \omega : E(\Gamma) \to \C \  \Big| \sum_{e | e^+ = v} \omega(e) = \sum_{e | e^- = v} \omega(e) \text{ for all }  v \in V(\Gamma) \right\} .$$

It is easy to see that $\dim_\C\Omega(\Gamma) = g'$. There exists a positive definite Hermitian pairing on $\Omega(\Gamma)$ given by
\begin{equation}
 \label{eqnMetricGraphPairing}
\langle \omega, \omega' \rangle = \sum_{e}\omega(e)\overline{\omega'(e)}l_e.
\end{equation} 

This Hermitian pairing should be thought of as the analogue of \eqref{eqnHermitianPairing} for metric graphs.

\begin{prop}[Theorem 5.10, Theorem 6.4 in \cite{BF11}]
\label{propBF11}
  Let $\omega_1,\dots,\omega_{g'}$ be an orthonormal basis of $\Omega(\Gamma)$ with respect to the Hermitian pairing \eqref{eqnMetricGraphPairing}. Let $r_e$ denote the resistance between $e^-$ and $e^+$ in the graph obtained by removing the interior of the edge $e$ from $\Gamma$.  Then, $$\sum_{i = 1}^{g'}\left|\omega_i(e) \right|^2 = \frac{1}{l_e + r_e}.$$
\end{prop}
\begin{proof}
  Translating to the notation used by \cite{BF11}, we have that $$\sum_{i=1}^{g'}|\omega_i(e)|^2 = \left\|\frac{1}{l_e} \int_{e} \right\|^2_{\mathrm{L}^2}.$$
  Using \cite[Theorem 5.10]{BF11}, we get that
  $$ \left\|\frac{1}{l_e} \int_{e} \right\|^2_{\mathrm{L}^2} = \frac{F(e)}{l_e} ,$$
  where $F(e)$ is the Foster coefficient defined by Baker and Faber. Now, \cite[Theorem 6.4]{BF11} tells us that
  $$ \frac{F(e)}{l_e} = \frac{1}{l_e + r_e} .$$
\end{proof}

\subsection{Relation between the residues and the dual graph}
\label{subsecRelationBetweenResidues}
Let $\psi_1,\dots,\psi_g$ be a basis of $H^0(\X_{0,\red},\omega_{\X_{0,\red}})$. Let $g' = g(\Gamma_\X)$.
Following the discussion in Section \ref{subsecOmegaX0Red}, we may assume that $\psi_{g'+1},\dots,\psi_g$ are holomorphic i.e.~have zero residues at all nodal points in $\X_{0,\red}$. 

Note that the residues of $\psi_1,\dots,\psi_{g'}$ at the points in $\tilde{\X_{0,\red}}$ that lie over nodes in $\X_{0,\red}$ cannot be arbitrary; they must satisfy the following constraints
\begin{itemize}
\item The residue theorem ensures that the sum of the residues of $\psi_i$ is zero on every irreducible component of $\tilde{\X_{0,\red}}$ for all $1 \leq i \leq g'$.
\item If $P'$ and $P''$ are points in $\tilde{\X_{0,\red}}$ that map to a node $P$ in $\X_0$, then the residues of $\psi_{i}$ at $P'$ and $P''$ sum to zero for all $1 \leq i \leq g'$. 
\end{itemize}

Now pick an arbitrary orientation for each of the edges of $\Gamma_\X$. For an edge $e$, let $e^-$ and $e^+$ denote the initial and the final vertex respectively. For each $\psi_i$ and a node $P \in \X_{0,\red}$, let $C^{P}_i$ denote the residue of $\psi_i$ at the point that lies over $P$ in the irreducible component associated to $e^{-}_{P}$.
The data of the residues of $\psi_i$ defines an element $\omega_i \in \Omega(\Gamma_\X)$ by $\psi_i \mapsto (e_P \mapsto C_{i}^P)$.
Conversely, given $\omega \in \Omega(\Gamma_\X)$, we can get a $\psi \in H^0(\X_{0,\red},\omega_{\X_{0,\red}})$ using the residue theorem. Such an element is uniquely determined up to an element that has no poles on $\tilde{\X_0}$ i.e.~up to an element in the linear span of $\psi_{g'+1},\dots,\psi_g$. 

Summarizing, we have the following short exact sequence of complex vector spaces.
\begin{equation}
\label{eqnSESVS1}
0 \to H^0(\tilde{\X_{0,\red}}, \omega_{\tilde{\X_{0,\red}}}) \to H^0(\X_{0,\red},\omega_{\X_0,\red}) \to \Omega(\Gamma_\X) \to 0
\end{equation}

\begin{lem}
\label{lemHermitianPairing}
We can pick a basis $\psi_1,\dots,\psi_{g}$ of $H^0(\X_{0,\red},\omega_{\X_{0,\red}})$ such that $$\sum_{P} C_j^P \overline{C_k^P}l_{e_P} = \delta_{jk}$$
for all $1 \leq j,k \leq g'$ and $$\int_{\tilde{\X_{0,\red}}} \psi_{j} \wedge \overline{\psi_{k}} = \delta_{jk}$$ for $g' + 1\leq j,k \leq g$.

\end{lem}
\begin{proof}
  We have a positive definite Hermitian pairing on $H^0(\tilde{\X_{0,\red}},\omega_{\tilde{\X_{0,\red}}}) = \bigoplus_i H^0(E_i,\omega_{E_i})$ given by the Hermitian pairing on each direct summand. We pick $\psi_{g'+1},\dots,\psi_g$ to be orthonormal with respect to this pairing.
  
 We pick $\psi_1,\dots,\psi_{g'}$ so that the induced $\omega_1,\dots,\omega_{g'} \in \Omega(\Gamma)$ form an orthonormal basis with respect to the pairing \eqref{eqnMetricGraphPairing}.
\end{proof}

It follows immediately from Proposition \ref{propBF11} that for a node $P \in \X_0$ and for the choice of $\psi_i$'s in Lemma \ref{lemHermitianPairing}, $$\sum_{i=1}^{g'} |C_i^P|^2 = \frac{1}{l_{e_P} + r_{e_P}},$$ which is the coefficient of $dx|_{e_P}$ that shows up in the Zhang measure.

\subsection{Relation between $\omega_{\X_{0,\red}}$  and the canonical bundle on $\X$}
\label{subsecRelationBetweenOmegaX0redAndTheCanonicalBundle}
Let $\X$ be an nc model of $X$. 
Let $\omega_{\X} \simeq \Omega_{\X}^2$ denote the canonical line bundle of $\X$ i.e.~the sheaf of 2-forms on $\X$. Note that we have an isomorphism $\Omega_{\X/\D} \simeq \omega_{\X}$ between the sheaf of relative holomorphic $1$-forms and the canonical line bundle. This isomorphism is given by `unwedging $dt$', where $t$ is the coordinate on $\D$.

Note that $\X_0, \X_{0,\red}$ are Cartier divisors on $\X$ and we can consider the line bundle $$\L := \omega_{\X}(-\X_0 + \X_{0,\red}) := \omega_{\X} \otimes_{\O_\X} \O_{\X}(-\X_0+\X_{0,\red}).$$

Since $\X_0 = \div(t)$ is a principal divisor, we have a canonical isomorphism $\L \simeq \omega_{\X}(\X_{0,\red})$

For $t \in \D^*$, note that $$\L|_{X_t} \simeq \omega_{\X}(\X_{0,\red})|_{X_t} \simeq \omega_{\X}|_{X_t} \simeq \Omega_{\X/\D}|_{X_t} \simeq \omega_{X_t}.$$

For the central fiber, we can use adjunction formula \cite[9.1.37]{Liu02} to conclude that
$$ \L|_{\X_{0,\red}} \simeq \omega_{\X_0}(\X_{0,\red})|_{\X_{0,\red}} \simeq \omega_{\X_{0,\red}} . $$

\begin{lem}
 \label{lemBlowupIsoOnH0Omega}
  Let $\X'$ be a nc model obtained from $\X$ by a single blowup at a closed point in $\X_{0}$. Let $q_0 : \X'_{0,\red} \to \X_{0,\red}$ be the map induced by the blowup map $q : \X' \to \X$. Then, we have an isomorphism obtained by ``pulling back differential forms''.
  $$q_0^* :  H^0(\X_{0,\red},\omega_{\X_{0,\red}}) \xrightarrow{\sim} H^0(\X'_{0,\red},\omega_{\X'_{0,\red}}) .$$
\end{lem}
\begin{proof}
We first describe the map $q_0^*$. To do this, we use a few elementary facts regarding blowups. Let $E$ denote the exceptional divisor and let $b$ denote its multiplicity in $q^*(\X_{0,\red})$. Note that $b$ is either 1 or 2, depending on whether we blowup at a smooth or at a nodal point in $\X_0$.
\begin{eqnarray*}
  q^*(\omega_{\X}) \otimes \O_{\X'}(E)= \omega_{\X'}\\
  q^*(\X_{0,\red}) = \X_{0,\red}' + (b-1)E 
\end{eqnarray*}

Using the above facts, we conclude that $$q^*(\omega_{\X}(\X_{0,\red})) = \omega_{\X'}(\X_{0,\red}' - E + (b-1)E).$$

Restricting the above equation to $\X_{0,\red}'$, we get
$$ q_0^*(\omega_{\X_{0,\red}}) = \omega_{\X'_{0,\red}}((b-2)E)  .$$

The following composition is the map $q_0^*$ and we claim that it is an isomorphism.
$$H^0(\X_{0,\red},\omega_{\X_{0,\red}}) \to H^0(\X'_{0,\red},q^*(\omega_{\X_0,\red})) \to H^0(\X'_{0,\red},\omega_{\X'_{0,\red}}).$$
Note that both $H^0(\X_{0,\red},\omega_{\X_{0,\red}})$ and $ H^0(\X'_{0,\red},\omega_{\X'_{0,\red}})$ are vector spaces of dimension $g$. So, it is enough to show that $q_0^*$ is injective. Note that any section $\psi \in H^0(\X_{0,\red},\omega_{\X_{0,\red}})$ is determined by all restrictions $\psi|_{E_i}$ for all the irreducible components $E_i$ of $\tilde{\X_{0,\red}}$. Also, $q_0^*\psi|_{{E_i'}} = \psi|_{E_i}$, where $E_i'$ is the strict transform of $E_i$. Thus, $\psi$ is determined by $q_0^*\psi$ and the map $q^*_0$ is injective.
\end{proof}

\begin{lem}
\label{lemGrauert}
Suppose that $X$ has semistable reduction. Let $\X$ be an nc model of $X$. Then there exist an $ r \in (0,1)$ and 2-forms $\theta_1,\dots,\theta_g \in H^0(r\D,\omega_\X(-\X_0 + \X_{0,\red}))$  such that $\theta_1|_{X_t},\dots,\theta_g|_{X_t}$ is a basis of $H^0(X_t, \omega_{X_t})$ for all $t \in r\D^*$  and $\theta_1|_{\X_{0,\red}},\dots,\theta_g|_{\X_{0,\red}}$ is a basis of $H^0(\X_0,\omega_{\X_{0,\red}})$. 
\end{lem}
\begin{proof}
As above, let $\L$ denote the line bundle $\omega_{\X}(-\X_0 + \X_{0,\red})$.

First suppose that $\X$ is a semistable model of $X$. Then $\X_{0,\red} = \X_0$ and we have that $$\dim_\C H^0(\X_0,\omega_{\X_0}) = \dim_\C H^0(\X_{0,},\omega_{\X_{0,\red}}) = g = \dim_\C H^0(X_t,\omega_{X_t})$$ for all $t \in \D^*$ and thus the dimension of $H^0(\X_t, \L|_{\X_t})$ for all $t \in \D$ remains constant. By a theorem of Grauert \cite{Gra60} (see \cite[Cor 3.12.19]{Har77} for an algebraic version), we get that  $\pi_*(\L)$ is a locally free sheaf and its fiber over 0, $\pi_*(\L)_{(0)}$, is isomorphic to $H^0(\X_{0},\omega_{\X_{0}})$. Now we pick $\theta_1,\dots,\theta_g\in \pi_*(\L)_{(0)}$ that map to a basis in $H^0(\X_{0},\omega_{\X_{0}})$. Then there exists an $0 < r < 1$ and $\tilde{\theta_1},\dots,\tilde{\theta_g} \in H^0(r\D,\L)$ which restrict to  a basis $\psi_1,\dots,\psi_g$ of $H^0(\X_{0,\red},\omega_{\X_{0,\red}})$. Since being linearly independent is an open condition, we may pick a smaller $r$ so that $\theta_1,\dots,\theta_g$ remain linearly independent (and hence form a basis) after restricting to $X_t$ for $|t| \ll r$. This completes the proof when $\X$ is a semistable model of $X$.

Any nc model can be obtained from the minimal nc model by a sequence of blowups at closed points in the central fiber. By induction, we may reduce to the proof to the case of a single blowup. 

Suppose now that the result is true for a nc model $\X$; we would like to prove the result for a nc model $\X'$ obtained by a single blowup $q: \X' \to \X$ at a closed point in $\X_0$. Let $E$ denote the exceptional divisor of the blowup.

Let $\theta_1,\dots, \theta_g$ be sections of $\omega_{\X_0}(-\X_0 + \X_{0,\red})$ in a neighborhood of $\X_{0,\red}$ that satisfy the required conditions for the model $\X$. We claim that $q^*\theta_1,\dots,q^*\theta_g$ satisfy the required conditions for $\X'$, where $q^*$ denotes the usual pullback of differential forms.

Since $q|_{X}$ is an isomorphism, it is clear that $(q^*\theta_1)|_{X_t},\dots,(q^*\theta_g)|_{X_t}$ form a basis of $H^0(X_t,\omega_{X_t})$. The fact that $(q^*\theta_1)|_{\X'_{0,\red}},\dots,(q^*\theta_g)|_{\X'_{0,\red}}$ is also a basis follows by Lemma \ref{lemBlowupIsoOnH0Omega}. 
\end{proof} 




\section{Asymptotics}
\label{secAsymptotics}
In this section, we compute some asymptotics to describe the Bergman measure in terms on $\theta_1,\dots,\theta_g$.
Suppose that $X$ has semistable reduction. We pick an nc model $\X$ of $X$. Let $\Gamma = \Gamma_{\X}$ denote its dual graph and let $g' = g(\Gamma)$.  Then, we have that $g = g' + \sum_{v \in V(\Gamma)} g(v)$. By Lemma \ref{lemGrauert}, we can find two-forms $\theta_1,\dots,\theta_g$ defined in a neighborhood of $\X_0$ such that their restrictions form a basis of $H^0(X_t,\omega_{X_t})$ and $H^0(\X_{0,\red},\omega_{\X_{0,\red}})$ for all $t \in \D^*$. Let  $\psi_i = \theta_i|_{\X_{0,\red}}$. After applying a (complex) linear transformation to $\theta_i$'s, we may assume that the $\psi_i$ satisfy the conditions in Lemma \ref{lemHermitianPairing}.

\subsection{Relating $\theta_i$, $\theta_{i,t}$ and $\psi_{i}$}
\label{subsecRelatingThetaThetaITAndPsi}
For doing computations, we would like to express $\theta_{i,t} := \theta_{i}|_{X_t}$ and $\psi_i$ explicitly in terms of $\theta_{i}$ in a local coordinate chart.

Let $U$ be a coordinate chart adapted to an irreducible component $E \subset \X_0$ of multiplicity $a$. Let $(z,w)$ be coordinates on $U$ such that $E = \{z = 0\}$ and $t = z^a$. Then, $\theta_i$ must vanish along $E$ to the order $a-1$ and has a power series expansion of the form $$\theta_i = \sum_{\alpha \geq a -1, \beta \in \N}c^{(i)}_{\alpha \beta} z^\alpha w^\beta dw \wedge dz.$$ Then, $\theta_{i,t}$ is just obtained by `unwedging $dt$'. To do this, note that $dt = a z^{a-1}dz$ and thus 
$$\theta_{i,t} = \sum_{\alpha \geq a - 1, \beta \in \N}\frac{c^{(i)}_{\alpha \beta}}{a} t^{\frac{\alpha-a+1}{a}} w^{\beta}dw.$$ Here we think of the coordinates on $X_t$ as being given by $0 < |w| < 1$. Taking the $a$-th root of $t$ corresponds to the fact that $U \cap X_t$ is disconnected and has $a$ connected components. Choosing a connected component corresponds to choosing an $a$-th root of $t$. Note that we will be somewhat sloppy while writing fractional powers of $t$. This should be interpreted as being true in a small enough chart where the roots are well defined. 

Tracing through the isomorphism in Section \ref{subsecRelationBetweenOmegaX0redAndTheCanonicalBundle}, we see that $\psi_i$ is obtained from $\theta_i$ by getting rid of $z^{a-1}dz$ and then setting $z = 0$. Thus, $$\psi_i = \sum_{\beta \in \N} c_{a-1,\beta}^{(i)}w^\beta dw$$
and we see that $\lim_{t \to 0} \theta_{i,t} = \frac{1}{a}\psi_i$, where the limit is taken pointwise as a function of $w$. 

Now consider a coordinate chart $U$ adapted to a node $P = E_1 \cap E_2$. We allow the possibility that $E_1$ and $E_2$ correspond to two local branches of the same irreducible component. Let the coordinates on $U$ be $(z,w)$ such that $|z|,|w| < 1$, $E_1 = \{ z = 0 \}$, $E_2 = \{ w = 0\}$ and $t = z^aw^b$. 
On $X_t \cap U$, we can either use the local coordinate $z$ with $|t|^{1/a} < |z| < 1$ or the coordinate $w$ with $|t|^{1/b
} < |w| < 1$. We also have coordinates $w$ on $E_1 \cap U$ for $|w| < 1$  and coordinates $z$ on $E_2 \cap U$ for $|z| < 1$. Also note that $X_t \cap U \to \{ w \in \D^* \mid  |t|^{1/b} < |w| < 1 \}$ is a $a$-sheeted cover with the fibers corresponding to choosing an $a$-th root to determine $z = (\frac{t}{w^b})^{1/a}$. Since $\theta_i$ must vanish along $E_1$ to order $a-1$ and along $E_2$ to order $b-1$, we can write
$$\theta_i = \sum_{\alpha \geq a-1, \beta \geq b -1 } c_{\alpha, \beta}^{(i)} z^\alpha w^\beta dw \wedge dz$$ on $U$.
Let us compute $\theta_{i,t}$ and $\psi_i$ in the $w$-coordinates. Using $dt = az^{a-1}w^bdz + bz^aw^{b-1}dw$, we have that
$$\theta_{i,t} = \sum_{\alpha \geq a-1, \beta \geq b -1 } \frac{c_{\alpha, \beta}^{(i)}}{a} \left( \frac{t}{w^b} \right)^{\frac{\alpha - a - 1}{a}} w^{\beta-b} dw .$$
To obtain $\psi_i$ on $E_1$, we need to get rid of $z^{a-1}w^{b-1}d(zw)$ and set $z = 0$. This gives us 
$$\psi_i =  \sum_{ \beta \geq b - 1} c_{a-1, \beta}^{(i)} w^{\beta - b } dw.$$

Similarly, we can compute $\theta_{i,t}$ in the $z$ coordinates and can obtain $\psi_i$ on $E_2$.

Once again we see that $\lim_{t \to 0} \theta_{i,t} = \frac{1}{a}\psi_i$ for fixed $w$ and $\lim_{t \to 0} \theta_{i,t} = \frac{1}{b} \psi_i$ for fixed $z$.

From the local description, we also see that $\psi_i|_{E_1}$ has a simple pole at $P$ with residue $c_{a-1,b-1}^{(i)}$ and $\psi_i|_{E_2}$ has a simple pole at $P$ with residue $-c_{a-1,b-1}^{(i)}$. Set $C_i^P = c^{(i)}_{a-1,b-1}$ for ease of notation. 
 


\subsection{Bergman measure in terms of $\theta_1,\dots,\theta_g$}
For $t \in \D^*$, let $A(t)$ be the complex $g \times g$ skew-Hermitian matrix with $(j,k)$-th entries $$A(t)_{j,k} = (\theta_{j,t},\theta_{k,t}) = \frac{i}{2}\int_{X_t} \theta_{j,t} \wedge \overline{\theta_{k,t}}.$$

Then the Bergman measure (as a $(1,1)$-form) on $X_t$ is given by$$\mu_t = \frac{i}{2}\sum_{j,k} (\overline{A(t)})^{-1}_{j,k} \theta_{j,t} \wedge \overline{\theta_{k,t}}.$$

To understand the asymptotics of the Bergman measure, we need to understand the entries of the matrix $A(t)^{-1}$ as $t \to 0$. We start by understanding the entries of the matrix $A(t)$. Similar asymptotics can be found in \cite[Proposition 4.1]{HJ96} and \cite[Equation (16.7)]{dJon19}.

\begin{lem}
\label{lemAsymptotics}
For a suitable choice of $\theta_1,\dots,\theta_g$, the matrix $A$ is of the form $$A =
\begin{pmatrix}
  B & C \\
  D & F
\end{pmatrix},$$
  
\noindent where $B =2\pi \log|t|^{-1} \I_{g'} + O(1)$ is  a $g' \times g'$ matrix, $C = O(1)$, $D = O(1)$ and $F = O(1)$ is a $(g-g') \times (g-g')$ matrix as $t \to 0$. Furthermore, we have that $\lim_{t \to 0} F(t) = \mathbf{I}_{g-g'}$  
\end{lem}

\begin{proof}
We pick $\theta_1,\dots, \theta_g$ such that $\psi_1,\dots,\psi_g$ satisfy the conditions of Lemma \ref{lemHermitianPairing}.

  By using partitions of unity, to understand the asymptotics of $\int_{X_t} \theta_{j,t}\wedge \overline{\theta_{k,t}}$ it is enough to understand the asymptotics of $\int_{U \cap X_t} \theta_{j,t} \wedge \overline{\theta_{k,t}}$ for an adapted coordinate chart $U$.

  Let $U$ be a coordinate chart adapted to an irreducible component $E \subset \X_0$ occurring with multiplicity $a$. For all $i$, we have that $\theta_{i,t} \to \frac{\psi_i}{a}$ as $t \to 0$. By shrinking $U$ if needed, we may further assume that $\theta_{i,t} \to \frac{\psi_i}{a}$ uniformly. Since all the $\psi_i$'s are bounded on $U \cap E$, by the dominated convergence theorem, we have that $\int_{U \cap X_t} \theta_{j,t} \wedge \overline{\theta_{k,t}} \to \int_{U \cap E} \psi_{j} \wedge \overline{\psi_{k}} $ as $t \to 0$ for all $1 \leq j,k \leq g$.

  If $U$ is a coordinate chart adapted to $P = E_1 \cap E_2$, then we break up the integral as $$\int_{U \cap X_t} \theta_{j,t} \wedge \overline{\theta_{k,t}} = \int_{|t|^{1/2a }<  |z| < 1} \theta_{j,t} \wedge \overline{\theta_{k,t}} + \int_{|t|^{1/2b} < |w| < 1} \theta_{j,t} \wedge \overline{\theta_{k,t}}.$$

On the set $|t|^{1/2b} < |w| < 1$, $\theta_{i,t}(w) = \frac{C^P_i}{aw} + O(1)$ (see the discussion in Section \ref{subsecRelatingThetaThetaITAndPsi}), where the $O(1)$ is with respect to $|w|$ as $|w| \to 0$ uniformly in $t$. Thus,
\begin{multline*}
\frac{i}{2}\int_{U \cap X_t}\theta_{j,t} \wedge \overline{\theta_{k,t}}
= \\
\frac{i}{2}\int_{|t|^{1/2b} < |w| < 1}  \frac{C^P_j\overline{C^P_k}}{a^2|w|^2} dw \wedge d\overline{w} + O(1)
= a \int_0^{2\pi} \int_{|t|^{1/2b}}^1  \frac{C^P_j\overline{C^P_k}}{a^2r} dr + O(1),  
\end{multline*}
where the second $O(1)$ is with respect to $r$ as $r \to 0$ and the factor $a$  appears on the right-hand side because $X_t \cap U \to \{ w \in \D \mid |t|^{1/b} < |w| < 1\}$ is an $a$-sheeted cover.
So,

$$\frac{i}{2}\int_{|t|^{1/2b} < |w| < 1} \theta_{j,t} \wedge \overline{\theta_{k,t}} =   \pi\frac{C^P_j \overline{C^P_{k}}}{ab}\log|t|^{-1} + O(1).$$

Using a similar computation in the $z$ coordinates and using $l_{e_P} = \frac{1}{ab}$ shows that
$$\frac{i}{2}\int_{X_t \cap U} \theta_{j,t} \wedge \overline{\theta_{k,t}} = 2\pi \frac{C^P_j\overline{C^P_k}}{ab}\log|t|^{-1} + O(1) = 2\pi C^P_j\overline{C^P_k}l_{e_P}\log|t|^{-1} + O(1).$$  

Summing up, we see that
\begin{equation}
\label{eqnIntThetaJK}
\frac{i}{2}\int_{X_t} \theta_{j,t} \wedge \overline{\theta_{k,t}} = 2\pi\sum_{\text{nodes} P \in \X_0} C^P_j\overline{C^P_k}l_{e_P} \log|t|^{-1} + O(1) .
\end{equation}

By the choice of $\theta_i$'s, $C^P_i = 0$ for all  $P$ and for all $i > g'$ giving the required asymptotics for the matrices $C,D$.
The asymptotics for $B$ follows from  $\sum_P C_j^P \overline{C_k^P}l_{e_P} = \delta_{jk} $.

To get the asymptotics for the matrix $F$, we need to analyze the $O(1)$-term in Equation (\ref{eqnIntThetaJK}). Recall that $\theta_{i,t} \to \frac{1}{a}\psi_i$ as $t \to 0$ for a fixed $w$ in the set $\{ w \in \D^* \mid |t|^{1/2b} < |w| < 1 \}$, and $\psi_i$ is bounded on $U \cap E_1$ for all $g' \leq i \leq g$. Thus, as $t \to 0$ we have that  
$$
\int_{|t|^{1/2b} < |w| < 1} \theta_{j,t} \wedge \overline{\theta_{k,t}} \to
\frac{1}{a}\int_{U \cap E_1} \psi_{j} \wedge \overline{\psi_{k}}.
$$

Thus, after applying a partition of unity argument, we get that 
$$
\int_{X_t} \theta_{j,t} \wedge \overline{\theta_{k,t}} \to
\sum_E \frac{1}{\mult_{\X_0} E} \int_{E} \psi_{j} \wedge \overline{\psi_{k}}.
$$

For $g'+1 \leq i \leq g$, $\psi_i$ is holomorphic on $\tilde{\X_0}$ on any irreducible component and therefore must be zero on any irreducible component with genus 0. Since $X$ has a semistable reduction, all positive genus irreducible components must occur with multiplicity $1$ (see Remark \ref{rmkMultiplicities}).  Thus, we further get that
$$
\int_{X_t} \theta_{j,t} \wedge \overline{\theta_{k,t}} \to \int_{\X_{0,\red}} \psi_{j} \wedge \overline{\psi_{k}}.
$$

By the choice of $\theta_i$'s, we have that $\frac{i}{2}\int_{\X_{0,\red}} \psi_{j} \wedge \overline{\psi_{k}} = \delta_{j,k}$ and thus we get the asymptotics for $F$. 
\end{proof}

Since $F \to \mathbf{I}_{g-g'}$ as $t \to 0$, $F$ is invertible for $|t|$ small enough. We now apply elementary row reduction operations to $(A,\I_g)$ to obtain the following result.

\begin{cor}
\label{corAsymptoticsAInverse}
For a suitable choice of $\theta_1,\dots,\theta_g$, the matrix $A^{-1}$ is of the form 
$$A^{-1} =
\begin{pmatrix}
  B' & C' \\
  D' & F'
\end{pmatrix}$$
where $$B' = \frac{1}{2\pi\log|t|^{-1}} \I_{g'} + O\left(\frac{1}{(\log|t|^{-1})^2}\right),$$
$$C',D' = O\left(\frac{1}{\log|t|^{-1}}\right) ,$$
$$F' = 
F^{-1} + O\left(\frac{1}{\log|t|^{-1}}\right)$$
and  $$\lim_{t \to 0} F'(t) = \mathbf{I}_{g-g'}.$$ \qed
\end{cor}



\section{Convergence Theorem}
\label{secConvergenceTheorem}
\subsection{Convergence on $\X^\hyb$}
\label{subsecConvergenceTheoremScriptX}
In this section, we prove Theorems \ref{mainThm2} and \ref{mainThm1}.

Suppose that $X$ has semistable reduction and let $\X$ is an nc model of $X$.  
Let $\mu_t$ denote the Bergman measure on $X_t$.

\subsection{Bergman measure on $\tilde{\X_{0,\red}}$}
By the Bergman measure on $\tilde{\X_{0,\red}}$, we mean the sum of the Bergman measures on all positive genus connected components of $\tilde{\X_{0,\red}}$.
The Bergman measure on $\tilde{\X_{0,\red}}$ is given by the two-form $\frac{i}{2}\sum_{i=g'+1}^{g} \psi_i \wedge \overline{\psi_i}$ 
Let $\tilde{\mu}_0$ denote the pushforward of the Bergman measure on $\tilde{\X_{0,\red}}$ to $\X_{0,\red}$. 
 
The following lemma gives the contribution of the Dirac mass on the vertices of $\Gamma_\X$ in the limiting measure. 
\begin{lem}
\label{lemGenusMassOnVertices}
   Consider an open set $U \subset \X$ adapted to an irreducible component $E$ of $\X_0$ of multiplicity $a$, i.e.~$U \cap \X_0 = E \cap U$ and there exist coordinates $z,w$ on $U$ with $|z|,|w| < 1$ such that $E \cap U = \{z = 0\}$ and the projection $U \to \D$ is given by $(z,w) \mapsto z^a$ and $|z|, |w| < 1$ on $U$. Let $\chi$ be a compactly supported continuous function on $U$. Then, as $t \to 0$, 
$$\int_{U \cap X_t} \chi \mu_t \to  \int_{U \cap E} \chi \tilde{\mu}_0.$$ 
 \end{lem}
\begin{proof}
Recall that $\mu_t = \frac{i}{2}\sum_{j,k} (\overline{A(t)})^{-1}_{j,k} \theta_{j,t} \wedge \overline{\theta_{k,t}}$.

If either $j \leq g'$ or $k \leq g'$, then $\overline{A(t)}^{-1}_{j,k} = O\left(\frac{1}{\log|t|^{-1}}\right)$, $\theta_{i,t}$ is bounded on $U$ and using Corollary \ref{corAsymptoticsAInverse},
$$\int_{U \cap X_t} \chi \cdot (\overline{A(t)})^{-1}_{j,k} \theta_{j,t} \wedge \theta_{k,t} = O\left(\frac{1}{\log|t|^{-1}}\right)$$ and hence goes to 0 as $t \to 0$.

  Therefore, we only need to worry about the terms for which $j,k > g'$. Recall that $\theta_{i,t} \to \frac{1}{a} \psi_i$ as $t \to 0$. Using a similar computation as in the proof of Lemma \ref{lemAsymptotics}, we get that 
  $$\lim_{t \to 0} \int_{U \cap X_t} \chi \mu_t =  \frac{1}{a }\int_{U \cap E} \chi \cdot \left(\sum_{j,k = g'+1}^{g}(\lim_{t \to 0}\overline{F'(t)_{j,k}}) \cdot \psi_j \wedge \overline{\psi_k}\right),$$ where $F'$ is the matrix from Corollary \ref{corAsymptoticsAInverse}.
  Since $\lim_{t \to 0} F'(t)  = \mathbf{I}_{g-g'}$, we get that
  $$\lim_{t \to 0} \int_{U \cap X_t} \chi \mu_t =  \int_{U \cap E} \chi \cdot \left(\frac{1}{a}\sum_{i=g'+1}^{g} \psi_i \wedge \overline{\psi_i}\right),$$

Using Remark \ref{rmkMultiplicities}, we get that $\psi_i = 0$ unless $a = 1$ for $g - g' + 1\leq i \leq g$. Thus, 
  $$\lim_{t \to 0} \int_{U \cap X_t} \chi \mu_t =  \int_{U \cap E} \chi \cdot \left(\sum_{i=g'+1}^{g} \psi_i \wedge \overline{\psi_i}\right),$$

The right-hand side is exactly $\int_{U \cap E} \chi \tilde{\mu}_0$.  
\end{proof}

In the following lemma, the first term on the right-hand side contributes to the Lebesgue measure in $\mu_0$ while the second term contributes to the Dirac mass in $\mu_0$. 
\begin{lem}
\label{lemKey}  
  Let $U \subset \X$ be an open set adapted to a node $P = E_1 \cap E_2$ in $\X_{0,\red}$, where $E_1, E_2$ are irreducible components of $\X_0$ with multiplicities $a,b$ respectively. Let $\chi$ be a compactly supported function on $U$ and let $f$ be a continuous function on $[0,\frac{1}{ab}]$.  Write the coordinates in $U$ as $z,w$ with $|z|,|w| <1$, $E_1 = \{z = 0\}$, $E_2 = \{w = 0\}$ and the projection to $\D$ given by $(z,w) \mapsto t = z^aw^b$. Let the coordinate on $X_t \cap U$ be  $w$ with $|t|^{1/2b} < |w| < 1$.
  Then, as $t \to 0$ we have that 
  \begin{multline*}
    \int_{|t|^{1/2b} < |w| < 1} \chi \cdot \left(f \circ \Log_U\left(\left(\frac{t}{w^b}\right)^{1/a},w\right)\right) \mu_t \to \\
    \chi(P)\cdot \frac{1}{l_{e_P} + r_{e_P}} \cdot \int_{0}^{1/2ab} f(u)du  + f(0) \cdot \int_{U \cap E_1} \chi \tilde{\mu}_0,
\end{multline*}
 where $\frac{1}{l_{e_P}+r_{e_P}}$ is the coefficient of $dx|_{e_P}$ in the Zhang measure. (See Section \ref{subsecZhangMeasure} for details.) 
\end{lem}
\begin{proof}
  

  To analyze the integral in the left-hand side of the lemma, we first note that
  $$\int_{|t|^{1/2b} < |w| < 1} \chi (f \circ \Log_U)\mu_t =
\frac{i}{2}  \sum_{j,k = 1}^g (\overline{A(t)})^{-1}_{j,k} \int_{|t|^{1/2b} < |w| < 1} \chi (f \circ \Log_U)  \theta_{j,t} \wedge \overline{\theta_{k,t}}$$ and then  analyze each of the terms. To do this, we break them up into three cases.
  \begin{itemize}
  \item ($j \leq g'$ and $k > g'$) or ($j > g'$ and $k < g'$)
  \item $j,k \leq g'$
    \item $j,k > g'$
  \end{itemize}

  We will prove that the first case does not contribute at all in the limit, the second case contributes the first term in right-hand side of the Lemma and the third case contributes the second term. 

  For the first case, note that if $j \leq g'$ and $k > g'$, then $(\overline{A(t)})^{-1}_{j,k} = O\left(\frac{1}{\log|t|^{-1}}\right)$.
  Since $|t| < |w|^{2b} $ on the region that we are integrating on, we see from the power series expansion that 
\begin{eqnarray*}
\theta_{j,t} =& \left(\frac{C_{j}}{aw} + O(1)\right) dw \\
\theta_{k,t} =& O(1) dw
\end{eqnarray*}

 where the $O(1)$ above are with respect to $|w|$ as $|w| \to 0$ uniformly in $t$. If we do a change of coordinates $w = re^{i\theta}$, we have that 
$$ 
\theta_{j,t} \wedge \overline{\theta_{k,t}} = O\left(\frac{1}{|w|}\right) dw  \wedge d\overline{w} = O(1) dr d\theta
$$
where the last $O(1)$ is with respect to $r$ as $r \to 0$. 
  
Thus we see that
$$(\overline{A(t)})^{-1}_{j,k}\int_{|t|^{1/2b} < |w| <1} \chi (f \circ \Log_U)  \theta_{j,t} \wedge \overline{\theta_{k,t}}   = O\left(\frac{1}{\log|t|^{-1}}\right).$$
By symmetry, the same holds when $j > g'$ and $k \leq g'$.

 Now consider the second case when $j,k \leq g'$. Then, $$ \overline{A(t)}^{-1}_{j,k} = \frac{\delta_{jk}}{2\pi\log|t|^{-1}} + O\left(\frac{1}{(\log|t|^{-1})^2}\right)$$ and
 $$\frac{i}{2}\theta_{j,t}\wedge \overline{\theta_{k,t}} = \frac{i}{2}\left(\frac{C_j\overline{C_k}}{a^2|w|^2} + O\left(\frac{1}{|w|}\right)\right) dw \wedge d\overline{w}
 = \left(\frac{C_j\overline{C_k}}{a^2r} +  O(1)\right) dr d\theta,$$

 First note that if $j \neq k$ and $j,k \leq g'$, then,
 \begin{eqnarray*}
   (\overline{A(t)})^{-1}_{j,k}\int_{|t|^{1/2b} < |w| < 1} \chi (f \circ \Log_U) \theta_{j,t} \wedge \overline{\theta_{k,t}}
   &= O\left(\frac{1}{(\log|t|^{-1})^2} \int_{|t|^{1/2b} < r < 1}  \frac{dr}{r}  \right) \\
   &= O\left( \frac{1}{\log|t|^{-1}} \right) \to 0 \quad \text{ as } t \to 0.
\end{eqnarray*}

 If $j\leq g' $, then,
\begin{multline*}
 \frac{i}{2} (\overline{A(t)})^{-1}_{j,j}\int_{|t|^{1/2b} < |w| < 1} \chi (f \circ \Log_U) \theta_{j,t} \wedge \overline{\theta_{j,t}} \\
  = \frac{1}{2a\pi\log|t|^{-1}}\int_{|t|^{1/2b} < r < 1} \chi f\left(\frac{\log r}{a\log|t|}\right)\frac{|C_j|^2 dr d\theta}{r} + O\left( \frac{1}{\log|t|^{-1}}\right).
\end{multline*}

 It is enough to figure out the limit of the integral on the right-hand side. To do this, consider a change of variable $u = \frac{\log r}{a\log |t|}$. Then, the integral on the right-hand side becomes
 $$ |C_j|^2 \int_0^{1/{2ab}} \int_{0}^{2\pi}\chi \cdot f(u) d\theta du.$$

 The integrand converges to $\chi(0,0) f(u)$ pointwise almost everywhere as $t \to 0$. Since the integrand is bounded, by the dominated convergence theorem, we have that $$\lim_{t \to 0} \frac{i}{2}(\overline{A(t)})^{-1}_{j,j}\int_{|t|^{1/2b} < |w| < 1} \chi (f \circ \Log_U) \theta_{j,t} \wedge \overline{\theta_{j,t}} = |C_j|^2 \chi(0,0) \int_{0}^{1/2ab} f(u) du .$$

 It follows from Proposition \ref{propBF11} that $\frac{1}{l_{e_P} + r_{e_P}} = \sum_{j=1}^{g'} |C_j|^2$. This gives us the first term on the right-hand side in the lemma.

 For the third case when $j,k > g'$, note that $\overline{A(t)}^{-1}_{j,k} = O(1)$ and $\theta_{j,t} \wedge \overline{\theta_{k,t}} = O(1)dw \wedge d\overline{w}$. Therefore, we can apply the dominated convergence theorem. The pointwise limit of the integrand as  $t \to 0$ is given by
 \begin{multline*}\overline{A(t)}_{j,k}^{-1} \cdot\chi\left(\frac{t}{w} ,w \right)  \cdot f\left( \frac{\log|w|}{a\log|t|} \right) \theta_{j,t} \wedge \overline{\theta_{k,t}}\to  \left(\lim_{t\to 0}\overline{F'(0)_{j,k}} \right) \cdot \chi(0,w) \cdot f(0) \cdot \frac{\psi_{j} \wedge \overline{\psi_k}}{a^2}.
 \end{multline*}

After interchanging the limit and the integral and using the fact that $\psi = 0$ unless $a = 1$ (see Remark \ref{rmkMultiplicities}), we get the second term on the right-hand side of the lemma. 
\end{proof}

\begin{cor}
  Let the notation be as in the Lemma \ref{lemKey}. Then, as $t \to 0$
  \begin{multline*}
    \int_{U \cap X_t} \chi (f \circ \Log_U) \mu_t = \\
    \chi(P) \frac{1}{l_{e_P} + r_{e_P}} \int_{0}^{1/ab} f(u) du + f(0) \int_{U \cap E_1} \chi \tilde{\mu_0} + f(1) \int_{U \cap E_1} \chi \tilde{\mu_0}.
  \end{multline*}
\end{cor}
\begin{proof}
  Note that $$\int_{U \cap X_t} \chi (f \circ \Log_U) \mu_t = \int_{|t|^{1/2a} < |z| <1} \chi (f \circ \Log_U) \mu_t + \int_{|t|^{1/2b} < |w| <1} \chi (f \circ \Log_U) \mu_t.$$
  Applying the previous lemma for the both the terms on the right-hand side, we are done. 
\end{proof}

\begin{cor}
  Let $V = \bigcup_{i} U_i$ be a neighborhood of $\X_0$ where $U_i$ are adapted coordinate charts. Let $\chi_i$ be a partition of unity with respect to the cover $U_i$.  Let $\Log_V = \sum_{i} \chi_i\Log_{U_i}$ be a global log function on $V$. Let $f$ be a continuous function on $\Gamma$. Then, as $t \to 0$,
  $$ \int_{X_t} (f \circ \Log_V) \mu_t  \to \int_{\Gamma}f \mu_{Zh}.$$
\end{cor}
\begin{proof}
Note that $$\int_{X_t} (f \circ \Log_V) \mu_t = \sum_i \int_{U_i \cap X_t} \chi_i (f \circ \Log_V) \mu_t.$$ 
  
Since $\Log_V - \Log_U = O\left( \frac{1}{\log|t|^{-1}} \right)$, as $t \to 0$, $$\int_{U_i \cap X_t} \chi (f \circ \Log_V - f \circ \Log_{U_i}) \mu_t \to 0.$$

Therefore, the limit we are interested in is the same as the limit of
$$ \sum_i \int_{U_i \cap X_t} \chi_i (f \circ \Log_{U_i}) \mu_t.$$

The result just follows from the using the previous two lemmas and using that $\int_{E}\tilde{\mu_0} = g(\tilde{E})$ for all irreducible components $E$ of $\X_0$. 
\end{proof}

The following Corollary is equivalent to Theorem \ref{mainThm2}.
\begin{cor}
\label{corConvergenceHybridSpace}
Let $h$ be a continuous function on $\X^\hyb$. Then, $\int h \mu_t \to \int h \mu_{Zh}$ as $t \to 0$. 
\end{cor}
\begin{proof}
  Let $f = h|_{\Gamma}$ and let $\tilde{h} = f \circ \Log_V$. By the previous lemma, the result is true for $\tilde{h}$ i.e.~$\int \tilde{h} \mu_t \to \int f \mu_t$ as $t \to 0$. Thus, it is enough to show that $\int (h - \tilde{h}) \mu_t \to 0$ as $t \to 0$. Pick $\epsilon > 0$. Since $h - \tilde{h} = 0$ on $\Gamma$ and since $h - \tilde{h}$ is continuous on $\X^\hyb$, there exists $0 < r < 1$ such that $|h - \tilde{h}| < \epsilon$ on all $\pi^{-1}(r \D)$. Thus, $|\int(h-\tilde{h}) \mu_t| \leq \epsilon g$ for all $|t| < r$.  Letting $\epsilon \to 0$, we get that  $\int (h - \tilde{h
  }) \mu_t \to 0$ as $t \to 0$. 
\end{proof}

\subsection{Extending the convergence to $X^\hyb$}
The convergence theorem on $\X^\hyb$ has the drawback that it depends on the choice of a normal crossing model. To remedy this, we consider the convergence on $X^\hyb$. Recall that $X^\hyb =  \varprojlim \X^\hyb$ and does not depend on the choice of an nc model of $X$.
We would like to extend the convergence to $X^\hyb$ by patching the convergence results for $\X^\hyb$ for all nc models $\X$ of $X$.

To do this, note that we have a canonical measure, $\mu_{0,X}$ on $X_\Ct^\an = \varprojlim \Gamma_\X$ induced by the Zhang measure on all the $\Gamma_\X$'s. This follows from the fact that if $\X' \geq \X$ and  we consider the retraction $\Gamma_{\X'} \to \Gamma_\X$, then the pushforward of the Zhang measure on $\Gamma_{\X'}$ to $\Gamma_\X$ is the same as the Zhang measure on $\Gamma_\X$. The compatibility of these measures thus prove Theorem \ref{mainThm1} in the case when $X$ has a semistable reduction.

\subsection{Ground field extension}
Now we need to treat the general case of the Theorem \ref{mainThm1} i.e.~the case when $X$ does not necessarily have semistable reduction. To do this, note that after performing a base change by $\D^* \to \D^*$ given by $u \mapsto u^n$, $X$ will have semistable reduction (see Section  \ref{subsecSemistableReductionAndMinimalModel}). So, we only need to understand what happens after we perform such a base change. 

So consider the map $\D^* \to \D^*$ given by $u \mapsto u^n$. Let $Y$ be the base change of $X$ along this map i.e.~we have a Cartesian diagram
$$
\begin{tikzcd}
  Y \ar[r] \ar[d]& X \ar[d] \\
  \D^* \ar[r, "u \mapsto u^n"] & \D^*
\end{tikzcd}.
$$
At the level of varieties, this corresponds to doing a base field extension $\Ct \to \Cu$ and $Y_\Cu = X_\Ct \times_{\Ct} \Spec\Cu$. Thus, we have a surjective map $Y^\an_\Cu \to X^\an_\Ct$. This map is compatible with $Y \to X$ in the sense that the map $Y^\hyb \to X^\hyb$ is continuous. 
We would like to relate the convergence of Bergman measures on $X_t$ to the convergence of Bergman measures on $Y_u$.

Note that if $X$ has a semistable model, then so does $Y$. To see this, pick the minimal nc model of $\X$ and base change it to get a model $\tilde{\Y}$ of $Y$. The model $\tilde{\Y}$ is not regular, but can be made regular after blowing up at each singular point $\lfloor \frac{n}{2}\rfloor$ times to get  a model $\Y$ of $Y$. Then $\Y$ is the minimal nc model of $Y$. It is easy to see that under the base change operation,  $\Gamma_\Y$ is obtained by scaling the lengths of all edges in $\Gamma_\X$ by a factor of $n$. Thus, we see that the Zhang measure on $\Gamma_Y$ is compatible with the Zhang measure on $\Gamma_X$ if assume that $X$ has a semistable reduction. Similarly, the Zhang measures on $Y^\an_\Cu$ and $X^\an_\Ct$ are compatible if we assume that $X$ has a semistable reduction.

Note that this not necessarily true if $X$ does not have semistable reduction. Starting with $X$, we can always perform a suitable base change so that $Y$ has semistable reduction. Let $\Y$ and $\X$ be nc models of $Y$ and $X$ respectively. Then, we have a map $Y^\an_\Cu \to X^\an_\Ct$, which gives rise to a local isometry $\Gamma_{\Y} \to \Gamma_{\X}$. Let $\mu_0$ be the Zhang measure on $Y^\an_\Cu$. Since the map $p : Y^\hyb \to X^\hyb$ is continuous, we get that the Bergman measure $\mu_t$ on $X_t$ converge to the pushforward measure $p_*(\mu_0)$ supported on the image of $\Gamma_\Y$ in $X_\Ct^\an$, thus completing the proof of  Theorem \ref{mainThm1}.

\section{Metrized curve complex hybrid space}
\label{secCCConvergence}
In this section, we prove Theorem \ref{mainThm3}. To do this, we first construct the metrized curve complex hybrid space. Let $X \to \D^*$ be a family of curves with semistable reduction.  Let $\X$ be an nc model of $X$. 

\subsection{Metrized curve complexes}
The \emph{metrized curve complex}, $\Delta_\CC(\X)$ , associated to $\X$ is a topological space which is obtained from $\tilde{\X_{0,\red}}$ by adding line segments joining the points that lie over the same nodal point. More precisely,
$$ \Delta_\CC(\X) =  \left( \tilde{\X_{0,\red}} \sqcup \bigsqcup_{e \in E(\Gamma_\X)}[0,l_{e}]\right)/\sim ,$$
where $P' \sim 0$ and $P'' \sim l_{e_P}$ for $P',P''  \in \tilde{\X_{0,\red}}$ that lie over a node $P$ and $0,l_{e_P} \in [0,l_{e_P}]$.
We call the image of an irreducible component of $\X_{0,\red}$ as a \emph{curve} in $\Delta_\CC(\X)$ and the image of $[0,l_{e}]$ as an \emph{edge} in $\Delta_\CC(\X)$.

We have a continuous map $\Delta_\CC(\X) \to \X_{0,\red}$ obtained by collapsing all the edges of $\Delta_\CC(\X)$ to the associated nodes. We also have a continuous map $\Delta_\CC(\X) \to \Gamma_\X$   obtained by collapsing the curves to the associated vertices.

We define a measure $\mu_\CC$ on $\Delta_\CC(\X)$ as follows. Let $\tilde{\mu_0}$ denote the Bergman measure on the positive genus components of $\tilde{\X_{0,\red}}$.
$$ \mu_\CC = \tilde{\mu_0} + \sum_{e \in E(\Gamma_\X)}\frac{dx|_{e}}{l_{e} + r_e} ,$$

where $\frac{1}{l_e + r_e}$ is the coefficient that shows up in the Zhang measure(see Section \ref{subsecZhangMeasure}) and $dx|_e$ is Lebesgue measure on the edge $e$ normalized to have length $l_e$.

We say that a point $Q \in \Delta_\CC(\X)$
\begin{itemize}
\item is in the \emph{interior of a curve} if it lies on a curve but not on an edge.  
\item is in the \emph{interior of an edge} if it lies on an edge but not on a curve.
\item is an \emph{intersection point} if lies on a curve as well as an edge.
\end{itemize}


\subsection{Curve complex hybrid space}
\label{subsecCurveComplexHybridSpace}
We define the \emph{curve complex hybrid space}, $\X^\hyb_\CC$, which as a set is given by
$$ \X^\hyb_\CC = X \sqcup \Delta_\CC(\X).$$
We declare the topology on $\X^\hyb$ to be the weakest topology satisfying the  following.
\begin{itemize}
\item $X \hookrightarrow \X_\CC^\hyb$ is an open immersion.
\item $\X_\CC^\hyb \to \X$ given by collapsing all edges in $\Delta_\CC(\X)$ is a continuous map.
\item $\X_\CC^\hyb \to \X^\hyb$ given by collapsing all curves in $\Delta_\CC(\X)$ to points is a continuous map.
\end{itemize}

We now describe a neighborhood basis of a point $Q \in \Delta_\CC(\X)$.
\begin{itemize}
\item If $Q$ is an interior point of a curve, then an adapted coordinate chart centered at $Q$ gives a neighborhood basis of $Q$.
\item If $Q$ is an interior point of an edge, let $P$ denote the node associated to the edge containing $Q$. Let $U$ be an adapted neighborhood chart around $P$. Let $\alpha,\beta \in e_P \simeq [0,l_{e_P}]$ such that $\alpha < Q < \beta$. If we view $(\alpha,\beta) \subset [0,l_{e_P}]$, then,
  $$ \{x \in U \setminus \X_0 \mid \Log_U(x) \in (\alpha,\beta) \} \cup (\alpha,\beta) $$
  is a neighborhood of $Q$. As we vary $U,\alpha$ and $\beta$, we get a neighborhood basis of $Q$.
\item
\label{Test}
If $Q$ is an intersection point, let $P$ denote the node associated to the edge containing $Q$. Let $U$ be an adapted coordinate chart centered at $P$ with coordinates $z,w$ with $|z|, |w| < 1$ such that the projection $\X \to \D$ is given by $(z,w) \mapsto z^aw^b$. Let $E_1 \cap U = \{ z = 0 \}$ and $E_2 \cap U = \{ w = 0 \}$, where $E_1,E_2$ are irreducible components of $\X_{0,\red}$. WLOG, assume that $\tilde{E_1}$ is the irreducible component of $\tilde{\X_{0,\red}}$ containing $Q$. We identify $e_P \simeq [0,\frac{1}{ab}]$ with $v_{E_1}$ identified with 0. Pick $0 < \epsilon < \frac{1}{2ab}$. Then,
  $$ \left\{ (z,w) \in U \setminus \X_0 \ \Big\vert \ \frac{\log |w|}{a\log |t|} < \epsilon  \right\} \cup (E_1 \cap U)  \cup [0,\epsilon)$$
is a neighborhood of $Q$. Varying $U$ and $\epsilon$, we get a neighborhood basis of $Q$. 
\end{itemize}

\subsection{Convergence of Bergman measures}
To show that the Bergman measures $\mu_t$ on $X_t$ converge to $\mu_\CC$ on $\Delta_\CC(\X)$, we can use a partition of unity argument to reduce the problem to studying the convergence on a neighborhood of each point in $\Delta_\CC(\X)$.

Consider a point $Q$ and consider a neighborhood $V$ of $Q$ as described at the end of Section \ref{subsecCurveComplexHybridSpace}. We need to show that the measures $\mu_t$ on $X_t \cap V$ converges weakly to $\mu_\CC$ on $\Delta_\CC(\X) \cap V$.

If $Q$ is an interior point of a curve, then this computation has been worked out in Lemma \ref{lemGenusMassOnVertices}. If $Q$ is an interior point of an edge, then a minor modification of Lemma \ref{lemKey} yields the result. So, it remains to prove the result in the result in the case when $Q$ is an intersection point.

\begin{lem}
  Let $Q$ be an intersection point in $\Delta_\CC(\X)$ and let $V$ be a neighbourhood of $Q$ in $\X^\hyb_\CC$ mentioned at the end of Section \ref{subsecCurveComplexHybridSpace}. Let $f$ be a continuous compactly supported function on $V$. Then, as $t \to 0$,
  $$ \int_{V \cap X_t} f \mu_t \to \int_{V \cap \Delta_\CC(\X)} f \mu_\CC.$$
\end{lem}
\begin{proof}
  Let $f$ be a compactly supported continuous function on $V$. Let $f_0 = f|_{V \cap \Delta_\CC(\X)}$.
  Note that $V \cap \Delta_\CC(\X)$ is homeomorphic to a \emph{half-dumbbell}
  $$D =  \{ (w,v) \in \D \times [0,\epsilon) \mid \text{Either } w = 0 \text{ or } v = 0  \} \subset \D \times [0,\epsilon).$$

Let $r : \D \times [0,\epsilon) \to D$ be a strong deformation retract.

Consider the compactly supported continuous function $h : V \to \R$ defined by
$$ h(z,w) = f_0\left(r\left(w,\frac{\log|w|}{a\log|t|}\right)\right)$$
for $(z,w) \in V \cap X$ and by
$$ h(x) = f_0(x)$$
for $x \in \Delta_\CC(\X)$.

We first prove that $$ \int_{V \cap X_t} h \mu_t \to \int_{V \cap \Delta_\CC(\X)} f_0 \mu_\CC.$$
To see this, recall the following facts from Sections \ref{secAsymptotics} and \ref{secConvergenceTheorem}:
$$\mu_t = \frac{i}{2} \sum_{j,k=1}^{g}\overline{A(t)^{-1}}_{j,k} \theta_{j,t} \wedge \overline{\theta_{k,t}},$$
where
$$\overline{A(t)^{-1}}_{j,k} \theta_{j,t} \wedge \overline{\theta_{k,t}} =
O\left( \frac{1}{\log|t|^{-1}} \right)$$
$\text{when either } j \leq g' \text{ and } k > g' \text{ or } j > g' \text{ and } k \leq g'$,
$$\overline{A(t)^{-1}}_{j,k} \theta_{j,t} \wedge \overline{\theta_{k,t}} =
\frac{1}{\log|t|^{-1}} \left( \frac{C_j \overline{C_k}\delta_{jk}}{a^2}\frac{dw \wedge d\overline{w}}{|w|^2} + O(1) \right)$$
when $j,k \leq g'$, and
$$\overline{A(t)^{-1}}_{j,k} \theta_{j,t} \wedge \overline{\theta_{k,t}} \to \delta_{j,k} \frac{\psi_j \wedge \overline{\psi_k}}{a^2}$$
when $j,k > g'$.

Thus, $\text{if either } j \leq g' \text{ and } k > g' \text{ or } j > g' \text{ and } k \leq g'$, then as $t \to 0$,
\begin{equation}
\label{eqnPfLemCCConvergence1}
 \overline{A(t)^{-1}}_{j,k}\int_{V \cap X_t} h  \theta_{j,t} \wedge \overline{\theta_{k,t}} \to 0 .
\end{equation}

We also get that
$$ \sum_{j,k=g'+1}^g \overline{A(t)^{-1}}_{j,k} \int_{V \cap X_t} h \theta_{j,t} \wedge \overline{\theta_{k,t}} \to \sum_{j=g'+1}^g \int_{U \cap E_1} \left( \lim_{t \to 0} h\left(\frac{t}{w},w\right) \right) \frac{\psi_j \wedge \overline{\psi_j}}{a^2} .$$

Note that $\lim_{t \to 0} h(\frac{t}{w},w) = \lim_{t \to 0} f_0(r(w,\frac{\log|w|}{a\log|t|})) = f_0(r(w,0)) = f_0(w,0)$. Note that $\psi = 0$ unless $a=1$ (see Remark \ref{rmkMultiplicities}). Also, recall that $\tilde{\mu_0} = \frac{i}{2}\sum_{j=g'+1}^g \psi_j \wedge \overline{\psi_j}$. Thus,
\begin{equation}
 \label{eqnPfLemCCConvergence2}
  \frac{i}{2}\sum_{j,k=g'+1}^g \overline{A(t)^{-1}}_{j,k} \int_{V \cap X_t} h \theta_{j,t} \wedge \overline{\theta_{k,t}} \to \int f_0 \tilde{\mu_0}.
\end{equation}

Now, it remains to consider the limit of 
$$ \frac{i}{2}\sum_{j,k=1}^{g'} \overline{A(t)^{-1}}_{j,k} \int_{V \cap X_t} h \theta_{j,t} \wedge \overline{\theta_{k,t}}  $$ 
as $t \to 0$.
But this is the same as the limit of
$$ \sum_{j=1}^{g'} \frac{|C_j|^2}{2\pi a^2\log|t|^{-1}} \int_{V \cap X_t} h \frac{dw \wedge d\overline{w}}{|w|^2}  = \sum_{j=1}^{g'} \frac{a|C_j|^2}{2\pi a^2\log|t|^{-1}} \int_{\{1 > |w| > |t|^{a\epsilon}  \}} h \frac{dw \wedge d\overline{w}}{|w|^2}  $$
as $t \to 0$. The factor $a$ appears on the right-hand side since $$V \cap X_t \to \{ w \in \D^* \mid |t|^{a\epsilon} < |w| <1 \}$$ is an $a$-sheeted cover.
Consider a change of variables $u = \frac{\log|w|}{a\log|t|}$ and $\theta = \arg(w)$. Then, the above integral is the same as
$$ \sum_{j=1}^{g'} \frac{|C_j|^2}{2\pi} \int_0^\epsilon \int_0^{2\pi} h\left(\frac{t}{|t|^{au}e^{i\theta}},|t|^{au}e^{i\theta}\right) d\theta du .$$

Note that
$$\lim_{t \to 0} h\left(\frac{t}{|t|^{au}e^{i\theta}},|t|^{au}e^{i\theta}\right) = \lim_{t \to 0} f_0\left(r\left(|t|^{au}e^{i\theta}, u\right)\right) = f_0(r(0,u)) = f_0(0,u)$$
almost everywhere for $u \in [0,\epsilon]$.
Also recall that $\sum_{j,k=1}^{g'} |C_j|^2 = \frac{1}{l_{e_P} + r_{e_P}}$. Thus, we get that
\begin{equation}
 \label{eqnPfLemCCConvergence3}
\frac{i}{2}\sum_{j,k=1}^{g'} \overline{A(t)^{-1}}_{j,k} \int_{V \cap X_t} h \theta_{j,t} \wedge \overline{\theta_{k,t}} \to \frac{1}{l_{e_P}+r_{e_P}}\int_0^{\epsilon} f_0 de. 
\end{equation}

Using Equations \eqref{eqnPfLemCCConvergence1}, \eqref{eqnPfLemCCConvergence2} and \eqref{eqnPfLemCCConvergence3}, we get that
$$ \int_{V \cap X_t} h \mu_t \to \int_{V \cap \Delta_\CC(\X)} f_0 \mu_\CC.$$

To show that $\int_{V \cap X_t} f \mu_t \to \int_{V \cap \Delta_\CC(\X)} f_0 \mu_\CC$, note that $h-f$ is a compactly supported continuous function on $V$ such that $(h - f)|_{\Delta_\CC(\X_0)} = 0$. Thus, given $\epsilon' > 0$, there exists an $t_0$ such that $|h-f| < \epsilon'$ on $V \cap X_t$ for $|t| <
|t_0|$. Thus,
$$\left|\int_{V\cap X_t} f \mu_t - \int_{V \cap X_t} h \mu_t\right| < \epsilon'g.$$
Taking $\epsilon' \to 0$, we get that
$$ \lim_{t \to 0}\int_{V \cap X_t} f \mu_t = \lim_{t \to 0}\int_{V \cap X_t} h \mu_t = \int_{V \cap \Delta_\CC(\X)} f_0 \mu_\CC.$$
\end{proof}

\bibliographystyle{halpha-abbrv}
\bibliography{biblio}
\end{document}